\documentclass[11pt]{article}
\usepackage[a4paper,margin=1in]{geometry}
\usepackage[T1]{fontenc}
\usepackage[utf8]{inputenc}
\usepackage{amsmath,amssymb,amsthm,mathtools}
\usepackage[colorlinks]{hyperref}

\newtheorem{theorem}{Theorem}[section]
\newtheorem{lemma}[theorem]{Lemma}
\newtheorem{proposition}[theorem]{Proposition}
\newtheorem{corollary}[theorem]{Corollary}
\theoremstyle{definition}
\newtheorem{definition}[theorem]{Definition}
\newtheorem{remark}[theorem]{Remark}

\title{$p$-variational capacity of interior condensers and geometric reduction by a fixed phase}
\author{Vicente Vergara\footnote{Department of Mathematics, Faculty of Physical and Mathematical Sciences, University of Concepci\'on, Concepci\'on, Chile. \texttt{vvergaraa@udec.cl}}}
\date{}

\begin{document}
	\maketitle

\begin{abstract}
	We study the $p$-variational capacity of interior condensers in a bounded open set $\Omega\subset\mathbb R^n$ when both plates are determined by a single phase $\theta:\Omega\to\mathbb R$ in $W^{1,\infty}(\Omega)$ through sublevel and superlevel sets. By restricting the admissible class to potentials of the form $u=v\circ\theta$ and applying the coarea formula, the problem reduces to a one-dimensional variational functional in the level variable, governed by an \textit{energy weight} that combines the gradient profile of $\theta$ and the geometry of its level sets. We obtain an explicit formula for the \textit{reduced problem}, construct an explicit optimal profile, and deduce an upper bound for the full geometric capacity by fibered restriction. In addition, we derive estimates for the energy weight in terms of the gradient and the size of the fibers, and analyze the local effect of critical levels on the integrability of the reduced resistance. Finally, we present symmetric models in which the fibered reduction coincides with the full geometric capacity and, in the linear case, a quantitative tangential obstruction to that exactness.
\end{abstract}



\section{Introduction}
\label{sec:introduccion-capacidad-fase}

The $p$-variational capacity of an \textit{interior condenser} in an open set $\Omega\subset\mathbb R^n$ measures the minimal cost of imposing a transition between two disjoint plates $E,F\subset\Omega$. In this work we study a geometrically organized subclass of this problem: the one in which both plates are determined by a single phase function $\theta:\overline\Omega\to\mathbb R$, Lipschitz on $\overline\Omega$. Given $a<b$, we introduce the sublevel and superlevel sets
\[
E_a:=\{x\in\Omega:\theta(x)\le a\},
\qquad
F_b:=\{x\in\Omega:\theta(x)\ge b\},
\]
and consider the full geometric capacity of the associated interior condenser,
\[
\operatorname{Cap}_{p,\Omega}(E_a,F_b).
\]
We shall call an \emph{interior condenser} a pair $(E,F)$ of subsets of $\Omega$ such that
\[
\overline E\subset \Omega,
\qquad
\overline F\subset \Omega,
\qquad
\overline E\cap\overline F=\varnothing.
\]
More precisely, if $(E,F)$ is an interior condenser in $\Omega$, we write
\[
\operatorname{Cap}_{p,\Omega}(E,F)
:=
\inf_{u\in \mathcal A_{E,F}(\Omega)}
\int_\Omega |\nabla u(x)|^p\,\mathrm{d}x,
\]
where
\[
\mathcal A_{E,F}(\Omega)
:=
\bigl\{
u\in W^{1,p}(\Omega):
u\le 0 \text{ q.e. on } E,\ 
u\ge 1 \text{ q.e. on } F
\bigr\}.
\]
Here and throughout, q.e.\ means quasi-everywhere with respect to the Sobolev $p$-capacity; see, for example, \cite{HeinonenKilpelainenMartio2006}. The quantity $\operatorname{Cap}_{p,\Omega}(E,F)$ thus belongs to the classical framework of variational condenser capacity; in particular, the modern theory of condenser capacity provides a natural background for this formulation, see, for example, \cite{BjornBjorn2024Condenser} and \cite[Def.~3.6]{GibaraKorteShanmugalingam2024}. Our goal, however, is not to develop that abstract theory, but rather to study the geometric reduction induced by a fixed phase in the interior Euclidean regime considered here.

We adopt the following perspective. Instead of starting from the capacitary solution of the geometric problem and then studying the geometry of its level sets, we fix from the outset a phase $\theta$ and ask which effective problem it induces on the level variable. From this perspective, a single phase simultaneously organizes the position of the plates, the fibered class of admissible functions, and the energy cost associated with the transition.

Several lines of work have studied the interaction between capacity, $p$-harmonic equations, and the geometry of level sets of capacitary functions.

In convex annular domains, the classical works of Lewis showed that the level-set geometry of $p$-harmonic solutions has a rigid structure; subsequently, Lewis and Nyström developed a fine analysis of $p$-harmonic functions and capacitary functions in Lipschitz and starlike annular domains \cite{Lewis1977ConvexRings,LewisNystrom2007RingDomains}. In a complementary direction, a more geometric literature has studied convexity, quasiconcavity, and curvature estimates for level sets of elliptic solutions in \emph{convex rings}, including results of Korevaar, Bianchini--Longinetti--Salani, Jost--Ma--Ou, and Zhang--Zhang \cite{Korevaar1990EllipticRingProblems,BianchiniLonginettiSalani2009ConvexRings,JostMaOu2012CurvatureLevelSets,ZhangZhang2013ConvexityLevelSets}. Our purpose here is different: not to develop a general theory of level sets of capacitary potentials, but to isolate the variational structure induced by a fixed phase and to understand which geometric information about $\theta$ controls the minimal energy cost within the corresponding fibered class.

The fundamental restriction is to consider admissible functions of the form
\[
u(x)=v(\theta(x)).
\]
Under this \textit{fibered ansatz}, the level variable $t=\theta(x)$ becomes the governing variable of the problem. The coarea formula then shows that the $p$-Dirichlet energy is not determined by the mere volumetric pushforward of $\theta$, but rather by the effective \textit{energy weight}
\begin{equation}\label{eq:peso-energetico-intro}
	A_{p,\theta}(t)
	=
	\int_{\Sigma_t\cap\Omega}
	|\nabla\theta(x)|^{p-1}\,\mathrm{d}\mathcal H^{n-1}(x),
	\qquad
	\Sigma_t:=\theta^{-1}(t).
\end{equation}
This weight records both the size of the fiber $\Sigma_t\cap\Omega$ and the contribution of the gradient law of $\theta$ to the energy cost of each level.

The restriction to \textit{fibered profiles} $\Sigma_t \cap \Omega$ thus leads to a weighted one-dimensional functional in the level variable. We introduce the \textit{reduced energy}
\[
\mathcal E_{p,\theta}(v;a,b)
:=
\int_a^b |v'(t)|^p A_{p,\theta}(t)\,\mathrm{d}t,
\]
defined on the natural class $W^{1,1}(a,b)$, with values in $[0,+\infty]$, and the \textit{reduced quantity} associated with the phase,
\[
\operatorname{cap}_{p,\theta}(a,b)
:=
\inf
\left\{
\mathcal E_{p,\theta}(v;a,b):
v\in W^{1,1}(a,b),\
v(a)=0,\
v(b)=1,\
\mathcal E_{p,\theta}(v;a,b)<\infty
\right\}.
\]
The fibered class built from Lipschitz profiles continues to play a central role in the passage to the full geometric capacity, but the quantity $\operatorname{cap}_{p,\theta}(a,b)$ is formulated intrinsically as a one-dimensional variational problem in the level variable.

The following result shows that, within the fibered class, the energy is expressed completely in terms of the weight $A_{p,\theta}$, which leads to an explicit formula for $\operatorname{cap}_{p,\theta}(a,b)$, to the identification of the optimal profile, and to its basic structural properties.

\begin{theorem}[Structure of the reduced problem]\label{thm:estructura-problema-reducido}
	Let $\theta:\overline\Omega\to\mathbb R$ be a Lipschitz phase on $\overline\Omega$, and let $a<b$ be admissible levels for the phase in the sense of Definition~\ref{def:niveles-admisibles-fase}. Then the fibered restriction $u=v\circ\theta$ reduces the variational problem to a one-dimensional functional governed by the energy weight $A_{p,\theta}$. In particular,
	\[
	\operatorname{cap}_{p,\theta}(a,b)
	=
	\left(
	\int_a^b A_{p,\theta}(t)^{-\frac{1}{p-1}}\,\mathrm{d}t
	\right)^{1-p},
	\]
	with the convention that the right-hand side equals $0$ when
	\[
	\int_a^b A_{p,\theta}(t)^{-\frac{1}{p-1}}\,\mathrm{d}t=+\infty.
	\]
	Moreover, if
	\[
	\int_a^b A_{p,\theta}(t)^{-\frac{1}{p-1}}\,\mathrm{d}t<\infty,
	\]
	then the reduced problem admits an explicit minimizer, given by
	\[
	v_*(t)=
	\frac{\displaystyle\int_a^t A_{p,\theta}(\tau)^{-\frac{1}{p-1}}\,\mathrm{d}\tau}
	{\displaystyle\int_a^b A_{p,\theta}(\tau)^{-\frac{1}{p-1}}\,\mathrm{d}\tau},
	\]
	and the minimum value coincides with
	\[
	\operatorname{cap}_{p,\theta}(a,b)
	=
	\left(
	\int_a^b A_{p,\theta}(t)^{-\frac{1}{p-1}}\,\mathrm{d}t
	\right)^{1-p}.
	\]
\end{theorem}

The second problem is to relate the reduced quantity to the \textit{full geometric capacity} $\operatorname{Cap}_{p,\Omega}(E_a,F_b)$ of the condenser induced by the phase. The restriction to fibered potentials not only defines an effective one-dimensional model, but also yields, without additional hypotheses, an upper bound for the original problem. The following result formalizes this basic comparison.

\begin{theorem}[Upper bound by fibered restriction]\label{thm:cota-superior-fibrada}
	Let $\theta:\overline\Omega\to\mathbb R$ be a Lipschitz phase on $\overline\Omega$, and let $a<b$ be admissible levels for the phase in the sense of Definition~\ref{def:niveles-admisibles-fase}. Then, for the associated interior condenser $(E_a, F_b)$, one has
	\begin{equation}\label{eq:cota-superior-fibrada-intro}
	\operatorname{Cap}_{p,\Omega}(E_a,F_b)
	\le
	\operatorname{cap}_{p,\theta}(a,b).
	\end{equation}
\end{theorem}

This inequality turns $\operatorname{cap}_{p,\theta}(a,b)$ into an effective control of the full geometric capacity. From this point on, the problem is to estimate the reduced quantity in terms of the geometry of the phase. In this direction, we show that two-sided bounds on $|\nabla\theta|$ and on the size of the fibers translate into two-sided estimates for the energy weight $A_{p,\theta}$, and consequently for $\operatorname{cap}_{p,\theta}(a,b)$. In this way one obtains explicit criteria for the reduced cost and, through the preceding inequality, corresponding bounds for the full geometric capacity.

More precisely, when the phase satisfies an \textit{eikonal structure}
\begin{equation}\label{eq:eikonal-eq}
	|\nabla\theta(x)|=\Gamma(\theta(x)),
\end{equation}
the weight factorizes as
\[
A_{p,\theta}(t)=\Gamma(t)^{p-1}S_\theta(t),
\qquad
S_\theta(t):=\mathcal H^{n-1}(\Sigma_t\cap\Omega),
\]
and the reduction is governed by the interaction between the gradient law $\Gamma$ and the fiber size $S_\theta$. In this sense, eikonal design provides a geometric principle that makes it possible to describe the reduced cost explicitly.

A natural problem, which lies beyond the scope of the present work, is the optimization of the fibered bound with respect to the phase. More precisely, once an admissible class $\mathcal F$ of phases is fixed, one is led to the design problem
\[
\inf_{\theta\in\mathcal F}\operatorname{cap}_{p,\theta}(a,b)
=
\inf_{\theta\in\mathcal F}
\left(
\int_a^b A_{p,\theta}(t)^{-1/(p-1)}\,\mathrm{d}t
\right)^{1-p},
\]
or equivalently the maximization of the associated reduced resistance. In this context, the present manuscript identifies the energy weight $A_{p,\theta}$ as the relevant quantity for formulating a variational phase-design problem.

Within the preceding variational framework, the article also analyzes the local effect of critical levels of $\theta$ on the integrability of the reduced resistance.

The following result identifies the local threshold governing the integrability of the reduced resistance near a critical level, in terms of the degeneracy of the gradient and the geometric collapse of the fibers. More precisely, for fixed parameters $\alpha,\nu\geq 0$, this threshold appears when, near a critical level $t_0$, one postulates a local profile of the form
\begin{equation}\label{eq:perfil-local-gradiente-intro}
	|\nabla\theta(x)|\asymp |\theta(x)-t_0|^\alpha,
	\qquad
	S_\theta(t)\asymp |t-t_0|^\nu.
\end{equation}
Consequently, by the definition of the energy weight \eqref{eq:peso-energetico-intro},
\begin{equation}\label{eq:perfil-local-peso-intro}
	A_{p,\theta}(t)\asymp |t-t_0|^{\alpha(p-1)+\nu}.
\end{equation}
Therefore,
\[
A_{p,\theta}(t)^{-1/(p-1)}
\asymp
|t-t_0|^{-\alpha-\nu/(p-1)},
\]
so that the local integrability of the reduced resistance is governed by the threshold
\begin{equation}\label{eq:umbral-local-introduccion}
	\alpha+\frac{\nu}{p-1}<1.
\end{equation}
The condition \eqref{eq:umbral-local-introduccion} determines the local integrability regime of the reduced resistance near $t_0$. Under a two-sided hypothesis on the energy weight on a one-sided interval $(t_0,t_0+\delta_0)$, one obtains the distinction between the transmissive regime and the supercritical regime for the reduced quantity associated with those levels.

\begin{theorem}[Local transmissibility threshold]
\label{thm:umbral-local-transmisibilidad}
Let $t_0\in\mathbb R$. Assume that there exist $\delta_0>0$, constants $c_1,c_2>0$, and parameters $\alpha\in[0,1)$, $\nu\ge 0$ such that
\[
c_1 (t-t_0)^{\alpha(p-1)+\nu}
\le
A_{p,\theta}(t)
\le
c_2 (t-t_0)^{\alpha(p-1)+\nu}
\]
for almost every $t\in (t_0,t_0+\delta_0)$. Then, for every $\delta\in(0,\delta_0)$, one has
\[
\int_{t_0}^{t_0+\delta} A_{p,\theta}(t)^{-1/(p-1)}\,\mathrm{d}t<\infty
\quad\Longleftrightarrow\quad
\alpha+\frac{\nu}{p-1}<1.
\]
In particular, if
\[
\alpha+\frac{\nu}{p-1}\ge 1,
\]
then
\[
\operatorname{cap}_{p,\theta}(t_0,t_0+\delta)=0
\qquad
\text{for every }\delta\in(0,\delta_0).
\]
\end{theorem}

In the supercritical regime, the divergence of the reduced resistance implies
\[
\operatorname{cap}_{p,\theta}(t_0,t_0+\delta)=0
\]
for every $\delta\in(0,\delta_0)$. If, in addition, $(t_0,t_0+\delta)$ is a pair of admissible levels for the phase, then Theorem~\ref{thm:cota-superior-fibrada} yields
\[
\operatorname{Cap}_{p,\Omega}\bigl(E_{t_0},F_{t_0+\delta}\bigr)=0.
\]
In particular, the vanishing of the reduced model is transmitted to the full geometric capacity whenever the condenser induced by those levels is well defined.

Section~\ref{sec:ejemplos} also contains an explicit family of degenerate phases in which this mechanism can be verified directly. In the monomial model one takes
\[
\Omega=(-L,L)\times D\subset\mathbb{R}^n,
\qquad
\theta(x)=|x_1|^\gamma,
\qquad \gamma>1,
\]
where $D\subset\mathbb{R}^{n-1}$ is a bounded open set. In this case, for $t>0$ close to $0$, the fibers $\Sigma_t$ consist of two planar sections,
\[
S_\theta(t)\asymp 1,
\qquad
|\nabla\theta|\asymp t^{\frac{\gamma-1}{\gamma}}
\quad\text{on }\Sigma_t,
\]
and therefore
\[
A_{p,\theta}(t)\asymp t^{\frac{(\gamma-1)(p-1)}{\gamma}}.
\]
Thus, the preceding criterion can be checked explicitly in this example.

In general, the fibered reduction does not coincide with the full geometric problem; see \eqref{eq:cota-superior-fibrada-intro}. The final examples clarify the scope of the preceding results. On the one hand, they exhibit explicit models compatible with the geometry of a fixed phase in which the fibered reduction loses no information. On the other hand, they show, in the linear case $p=2$, a quantitative tangential obstruction to that exactness outside that regime.

The paper is organized as follows. In Section~\ref{sec:setup-capacidad-fase} we fix the variational formulation of the relative capacity, of condensers induced by a phase, and of the associated energy weight. In Section~\ref{sec:reduccion-1d-capacidad} we develop the reduced problem and prove Theorem~\ref{thm:estructura-problema-reducido}. In Section~\ref{sec:comparabilidad-geometrica} we establish the upper bound by fibered restriction and prove Theorem~\ref{thm:cota-superior-fibrada}; in addition, we derive estimation criteria for the energy weight. In Section~\ref{sec:regimenes-criticos-theta} we study the local behavior of the energy weight near a given level, analyze the local integrability of the reduced resistance, and prove Theorem~\ref{thm:umbral-local-transmisibilidad}. Finally, in Section~\ref{sec:ejemplos} we present explicit models of exactness of the fibered reduction in symmetric cases and a quantitative tangential obstruction to fibered exactness in the linear case.

Throughout the paper, the notation
\[
f(t)\asymp g(t)
\qquad\text{as } t\to t_0
\]
means that there exist constants $c,C>0$ and a neighborhood of $t_0$ such that
\[
c\,g(t)\le f(t)\le C\,g(t)
\]
for every $t$ in that neighborhood. In particular, the notation $f\propto g$ will be reserved to indicate exact proportionality, that is, the existence of a constant $C>0$ such that $f=Cg$.
\section{Geometric setup and capacity between plates}
\label{sec:setup-capacidad-fase}

We work in a bounded open set $\Omega\subset\mathbb{R}^n$, with fixed $1<p<\infty$. Unless explicitly stated otherwise, the phase $\theta:\overline\Omega\to\mathbb{R}$ will be assumed to be Lipschitz on $\overline\Omega$. In particular, its restriction to $\Omega$ belongs to $W^{1,\infty}(\Omega)$. Whenever we use a classical regime, for example for pointwise identities or differential computations on levels, we will indicate this locally and explicitly, typically by means of a hypothesis of the form $\theta\in C^2(U)$ in the relevant region $U\subset\Omega$. With this convention, we fix the geometric notation associated with $\theta$ and recall the coarea formula in the form that we will use throughout the paper; see, for example, \cite{EvansGariepy2015}. For each $t\in\mathbb{R}$ we write
\[
\Sigma_t:=\{x\in\Omega:\theta(x)=t\},
\qquad
S_\theta(t):=\mathcal H^{n-1}(\Sigma_t\cap\Omega).
\]
Then, for every nonnegative measurable function $g$,
\begin{equation}
	\label{eq:coarea-capacidad}
	\int_\Omega g(x)\,|\nabla\theta(x)|\,\mathrm{d}x
	=
	\int_{\mathbb{R}}
	\left(
	\int_{\Sigma_t\cap\Omega} g\,\mathrm{d}\mathcal H^{n-1}
	\right)\mathrm{d}t.
\end{equation}

For the fibered reduction, we consider not only the geometric size of the fiber $S_\theta(t)$, but also the effective energy weight induced by the phase.

\begin{definition}[Energy weight]
	\label{def:peso-energetico}
	For almost every $t\in\mathbb{R}$ we define
	\[
	A_{p,\theta}(t)
	:=
	\int_{\Sigma_t\cap\Omega}
	|\nabla\theta|^{p-1}\,\mathrm{d}\mathcal H^{n-1}.
	\]
\end{definition}

Although a first coarea reduction in the level variable suggests the Lebesgue pushforward weight by $\theta$, whose formal density is given by
\[
w_\theta(t)
:=
\int_{\Sigma_t\cap\Omega}\frac{1}{|\nabla\theta|}\,\mathrm{d}\mathcal H^{n-1},
\]
the $p$-Dirichlet energy naturally induces the effective energy weight $A_{p,\theta}$. Consequently, the reduced problem is governed by $A_{p,\theta}$.

\subsection{Interior condensers}
\label{subsec:condensadores-interiores}
We start from an interior condenser formed by two disjoint subsets of $\Omega$; the associated variational quantity is the relative $p$-capacity in $\Omega$.

\begin{definition}[Interior condenser]
	\label{def:condensador-interior}
	Let $E,F\subset \Omega$. We shall say that $(E,F)$ is an \emph{interior condenser in $\Omega$} if
	\[
	\overline E\subset \Omega,\qquad \overline F\subset \Omega,\qquad \overline E\cap \overline F=\varnothing.
	\]
\end{definition}

This hypothesis ensures that the two plates are strictly separated inside $\Omega$, so that the variational problem between the values $0$ and $1$ is posed without interaction with the outer boundary $\partial\Omega$.

\subsection{Relative $p$-capacity}
\label{subsec:capacidad-relativa}

\begin{definition}[Admissible class]
	\label{def:clase-admisible-capacidad}
	Let $(E,F)$ be an interior condenser in $\Omega$. We define
	\[
	\mathcal A_{E,F}(\Omega)
	:=
	\bigl\{
	u\in W^{1,p}(\Omega):\
	u\le 0 \text{ q.e. on } E,\
	u\ge 1 \text{ q.e. on } F
	\bigr\}.
	\]
\end{definition}

Here and throughout, q.e.\ means quasi-everywhere with respect to the Sobolev $p$-capacity; see, for example, \cite{HeinonenKilpelainenMartio2006}.

\begin{definition}[Relative $p$-capacity]
	\label{def:capacidad-p-general}
	Let $(E,F)$ be an interior condenser in $\Omega$. We define
	\[
	\operatorname{Cap}_{p,\Omega}(E,F)
	:=
	\inf_{u\in \mathcal A_{E,F}(\Omega)}
	\int_\Omega |\nabla u(x)|^p\,\mathrm{d}x.
	\]
\end{definition}

This is the $n$-dimensional geometric quantity that we will compare with the one-dimensional reduced quantity associated with a fixed phase $\theta$.

\begin{remark}[Truncation of the admissible class]
	\label{rem:truncacion-clase-admisible}
	By truncation, the preceding infimum does not change if one restricts the admissible class to functions such that
	\[
	0\le u\le 1
	\qquad\text{a.e. in }\Omega.
	\]
	Indeed, if $u\in\mathcal A_{E,F}(\Omega)$ and we define
	\[
	T(u):=\min\{1,\max\{0,u\}\},
	\]
	then $T(u)\in\mathcal A_{E,F}(\Omega)$ and
	\[
	\int_\Omega |\nabla T(u)(x)|^p\,\mathrm{d}x
	\le
	\int_\Omega |\nabla u(x)|^p\,\mathrm{d}x.
	\]
\end{remark}

The quantity $\operatorname{Cap}_{p,\Omega}(E,F)$ belongs to the classical framework of variational condenser capacity. A general formulation of condenser capacity for two plates in metric spaces appears, for example, in \cite[Def. 3.6]{GibaraKorteShanmugalingam2024}. In the interior Euclidean regime considered here, this geometric problem is a specialization adapted to the fibered reduction by a fixed phase.

\subsection{Condensers induced by a phase}
\label{subsec:condensadores-por-fase}

Let $\theta:\overline\Omega\to\mathbb R$ be a Lipschitz phase on $\overline\Omega$. For $a<b$, we introduce the subsets
\begin{equation}
	\label{eq:def-Ea-Fb}
	E_a:=\{x\in \Omega \, : \theta\le a\},
	\qquad
	F_b:=\{x\in \Omega \, :\theta\ge b\}.
\end{equation}

\begin{definition}[Admissible levels for a phase]
	\label{def:niveles-admisibles-fase}
Let $\theta:\overline\Omega\to\mathbb R$ be a Lipschitz phase on $\overline\Omega$, and let $a<b$. We say that $(a,b)$ is a pair of \emph{admissible levels} for the phase $\theta$ if
	\[
	a<\theta(x)<b
	\qquad
	\text{for every }x\in\partial\Omega,
	\]
	and moreover
	\[
	E_a\neq\varnothing,
	\qquad
	F_b\neq\varnothing.
	\]
\end{definition}

\begin{proposition}
	\label{prop:niveles-admisibles-implican-condensador}
	If $(a,b)$ is a pair of admissible levels for $\theta$, then $(E_a,F_b)$ is an interior condenser in $\Omega$.
\end{proposition}

\begin{proof}
	Since $\theta$ is continuous on $\overline\Omega$, the sets
	\[
	E_a=\Omega\cap\theta^{-1}((-\infty,a]),
	\qquad
	F_b=\Omega\cap\theta^{-1}([b,\infty))
	\]
	are relatively closed in $\Omega$. The condition
	\[
	a<\theta(x)<b
	\qquad
	\text{for every }x\in\partial\Omega
	\]
	implies that no point of $\partial\Omega$ belongs to the closures of $E_a$ and $F_b$. Consequently,
	\[
	\overline{E_a}\subset\Omega,
	\qquad
	\overline{F_b}\subset\Omega.
	\]
	Moreover, since $a<b$, one has
	\[
	E_a\cap F_b=\varnothing.
	\]
	Since $E_a\subset\{\theta\le a\}$ and $F_b\subset\{\theta\ge b\}$, continuity of $\theta$ on $\overline\Omega$ also yields
	\[
	\overline{E_a}\subset\{\theta\le a\},
	\qquad
	\overline{F_b}\subset\{\theta\ge b\},
	\]
	and therefore
	\[
	\overline{E_a}\cap\overline{F_b}=\varnothing.
	\]
	Since both sets are nonempty by hypothesis, we conclude that $(E_a,F_b)$ is an interior condenser in $\Omega$.
\end{proof}

For such admissible pairs, we introduce the relative geometric capacity
\[
\operatorname{Cap}_{p,\Omega}(E_a,F_b).
\]

\subsection{Reduction by levels}

The reduction arises by restricting the admissible class to functions fibered by a fixed phase,
\[
u=v\circ\theta.
\]
In that case, the $n$-dimensional energy induces a one-dimensional functional in the level variable, governed by the energy weight $A_{p,\theta}(t)$. The next section develops this reduced formulation.
\section{One-dimensional reduction and proof of Theorem~\ref{thm:estructura-problema-reducido}}
\label{sec:reduccion-1d-capacidad}

In this section we develop the reduced problem associated with a fixed phase $\theta\in W^{1,\infty}(\Omega)$ and prove Theorem~\ref{thm:estructura-problema-reducido}. We start from the fibered ansatz $u=v\circ\theta$, identify the reduced energy governed by $A_{p,\theta}$, obtain an explicit formula for $\operatorname{cap}_{p,\theta}(a,b)$, and construct an explicit optimal profile.

\subsection{Fibered ansatz}
\label{subsec:ansatz-fibrado}

Let $\theta:\Omega\to\mathbb R$ be a phase in $W^{1,\infty}(\Omega)$, and let $a<b$ be admissible levels. Recall the notation
\[
E_a:=\{\theta\le a\},
\qquad
F_b:=\{\theta\ge b\}.
\]
We consider the subclass of admissible functions of the form
\begin{equation}
	\label{eq:ansatz-fibrado}
	u(x)=v(\theta(x)),
\end{equation}
where $v:\mathbb R\to\mathbb R$ is a scalar profile. If, in addition,
\[
v(t)=0 \quad \text{for } t\le a,
\qquad
v(t)=1 \quad \text{for } t\ge b,
\]
then $u=0$ on $E_a$ and $u=1$ on $F_b$. Consequently, $u$ satisfies the plate conditions of the condenser induced by the phase.

\subsection{Reduced energy}
\label{subsec:energia-reducida}

If $u=v\circ\theta$ is sufficiently regular, the chain rule and the coarea formula \eqref{eq:coarea-capacidad} lead to the formal identity
\begin{equation}
	\label{eq:energia-reducida-formal}
	\int_\Omega |\nabla(v\circ\theta)(x)|^p\,\mathrm{d}x
	=
	\int_{\mathbb R}
	|v'(t)|^p
	\left(
	\int_{\Sigma_t\cap\Omega}
	|\nabla\theta|^{p-1}\,\mathrm{d}\mathcal H^{n-1}
	\right)\mathrm{d}t
	=
	\int_{\mathbb R}|v'(t)|^p A_{p,\theta}(t)\,\mathrm{d}t.
\end{equation}
This leads us to introduce the following reduced functional.

\begin{definition}[Reduced functional]
	\label{def:funcional-reducido}
	For $a<b$ and for $v\in W^{1,1}(a,b)$, we define
	\[
	\mathcal{E}_{p,\theta}(v;a,b)
	:=
	\int_a^b |v'(t)|^p\,A_{p,\theta}(t)\,\mathrm{d}t,
	\]
	with values in $[0,+\infty]$.
\end{definition}

In one dimension, the class $W^{1,1}(a,b)$ coincides with that of absolutely continuous functions on $[a,b]$ modulo equality almost everywhere. Therefore, if $v\in W^{1,1}(a,b)$, the boundary values $v(a)$ and $v(b)$ are well defined for the absolutely continuous representative, and moreover one has
\[
v(b)-v(a)=\int_a^b v'(t)\,\mathrm{d}t.
\]

For the comparison with the geometric capacity, we will use Lipschitz profiles, since these allow one to construct fibered admissible functions in $\Omega$. However, the reduced variational quantity is defined on the natural class $W^{1,1}$ with finite energy.

\begin{lemma}[Truncated extension and Lipschitz composition]
	\label{lem:extension-truncada-composicion}
	Let $\theta\in W^{1,\infty}(\Omega)$ and let $v:[a,b]\to\mathbb R$ be Lipschitz with
	\[
	v(a)=0,\qquad v(b)=1.
	\]
	Define
	\[
	\widetilde v(t)=
	\begin{cases}
		0, & t\le a,\\
		v(t), & a<t<b,\\
		1, & t\ge b.
	\end{cases}
	\]
	Then $\widetilde v$ is globally Lipschitz on $\mathbb R$, the composition
	\[
	u:=\widetilde v\circ\theta
	\]
	belongs to $W^{1,p}(\Omega)$ and satisfies
	\[
	\nabla u(x)=\widetilde v'(\theta(x))\,\nabla\theta(x)
	\qquad
	\text{for a.e. }x\in\Omega.
	\]
\end{lemma}

\begin{proof}
Since $v$ is Lipschitz on $[a,b]$, the truncated extension $\widetilde v$ is Lipschitz on all of $\mathbb R$. Applying the chain rule for Sobolev--Lipschitz compositions to $\theta\in W^{1,\infty}(\Omega)$, we obtain
	\[
	u=\widetilde v\circ\theta\in W^{1,p}(\Omega),
	\]
	and
	\[
	\nabla u(x)=\widetilde v'(\theta(x))\,\nabla\theta(x)
	\qquad
	\text{for a.e. }x\in\Omega;
	\]
	see, for example, \cite{EvansGariepy2015,HeinonenKoskelaShanmugalingamTyson2015}.
\end{proof}

\begin{proposition}[Rigorous fibered energy identity for Lipschitz profiles]
	\label{prop:identidad-rigurosa-energia-fibrada}
	Let $\theta\in W^{1,\infty}(\Omega)$ and let $v:[a,b]\to\mathbb R$ be Lipschitz with
	\[
	v(a)=0,\qquad v(b)=1.
	\]
	If $\widetilde v$ and $u=\widetilde v\circ\theta$ are as in Lemma~\ref{lem:extension-truncada-composicion}, then
	\begin{equation}
		\label{eq:identidad-rigurosa-energia-fibrada}
		\int_\Omega |\nabla u(x)|^p\,\mathrm{d}x
		=
		\int_a^b |v'(t)|^p\,A_{p,\theta}(t)\,\mathrm{d}t.
	\end{equation}
\end{proposition}

\begin{proof}
	By Lemma~\ref{lem:extension-truncada-composicion},
	\[
	|\nabla u(x)|^p
	=
	|\widetilde v'(\theta(x))|^p\,|\nabla\theta(x)|^p
	\qquad
	\text{for a.e. }x\in\Omega.
	\]
	We apply \eqref{eq:coarea-capacidad} with
	\[
	g(x):=|\widetilde v'(\theta(x))|^p\,|\nabla\theta(x)|^{p-1},
	\]
	which is measurable and nonnegative. Applying the coarea formula for Lipschitz functions \cite{EvansGariepy2015,FedererGMT}, we obtain
	\[
	\int_\Omega |\nabla u(x)|^p\,\mathrm{d}x
	=
	\int_{\mathbb R}
	|\widetilde v'(t)|^p
	\left(
	\int_{\Sigma_t\cap\Omega}|\nabla\theta|^{p-1}\,\mathrm{d}\mathcal H^{n-1}
	\right)\mathrm{d}t.
	\]
	Since $\widetilde v'(t)=0$ for almost every $t\notin(a,b)$ and $\widetilde v'(t)=v'(t)$ for almost every $t\in(a,b)$, we conclude that
	\[
	\int_\Omega |\nabla u(x)|^p\,\mathrm{d}x
	=
	\int_a^b |v'(t)|^p\,A_{p,\theta}(t)\,\mathrm{d}t.
	\]
\end{proof}

\subsection{Reduced capacity}
\label{subsec:capacidad-reducida}
	
\begin{definition}[Reduced capacity]
	\label{def:capacidad-reducida}
	For $a<b$, we define
	\[
	\operatorname{cap}_{p,\theta}(a,b)
	:=
	\inf
	\left\{
	\mathcal{E}_{p,\theta}(v;a,b):
	v\in W^{1,1}(a,b),\ 
	v(a)=0,\ 
	v(b)=1,\ 
	\mathcal E_{p,\theta}(v;a,b)<\infty
	\right\}.
	\]
\end{definition}
	
\subsection{Exact formula for the reduced capacity}
	\label{subsec:formula-capacidad-reducida}

\paragraph{Proof of Theorem~\ref{thm:estructura-problema-reducido}.}

Proposition~\ref{prop:identidad-rigurosa-energia-fibrada} reduces the variational problem to a one-dimensional functional governed by the weight $A_{p,\theta}$. We now obtain the explicit formula for $\operatorname{cap}_{p,\theta}(a,b)$ and construct an explicit optimal profile.

Let
\[
B(t):=A_{p,\theta}(t)^{-\frac{1}{p-1}}.
\]
If $v$ is an admissible profile for the reduced problem, then
\[
v(a)=0,\qquad v(b)=1,
\]
and therefore
\[
1=\int_a^b v'(t)\,\mathrm{d}t
=
\int_a^b v'(t)\,A_{p,\theta}(t)^{1/p}B(t)^{1/q}\,\mathrm{d}t,
\]
where $q=\frac{p}{p-1}$. Moreover,
\[
B(t)^{1/q}
=
A_{p,\theta}(t)^{-\frac{1}{p}},
\]
and therefore
\[
A_{p,\theta}(t)^{1/p}B(t)^{1/q}=1.
\]
Applying Hölder's inequality with conjugate exponents $p$ and $q$, we obtain
\[
1
\le
\left(
\int_a^b |v'(t)|^p A_{p,\theta}(t)\,\mathrm{d}t
\right)^{1/p}
\left(
\int_a^b B(t)\,\mathrm{d}t
\right)^{1/q}.
\]
Therefore,
\[
\mathcal{E}_{p,\theta}(v;a,b)
\ge
\left(
\int_a^b B(t)\,\mathrm{d}t
\right)^{1-p}.
\]
Taking infima over all admissible profiles, we obtain
\begin{equation}
	\label{eq:cota-cantidad-reducida-inferior}
	\operatorname{cap}_{p,\theta}(a,b)
	\ge
	\left(
	\int_a^b B(t)\,\mathrm{d}t
	\right)^{1-p}.
\end{equation}

For the upper bound, for each $k\in\mathbb N$ define
\[
B_k(t):=\min\{B(t),k\},
\qquad
I_k:=\int_a^b B_k(t)\,\mathrm{d}t.
\]
Since $B_k\in L^\infty(a,b)$, the function
\[
v_k(t):=\frac{1}{I_k}\int_a^t B_k(s)\,\mathrm{d}s
\]
is Lipschitz on $[a,b]$ and satisfies
\[
v_k(a)=0,\qquad v_k(b)=1.
\]
Moreover,
\[
v_k'(t)=\frac{B_k(t)}{I_k}
\qquad\text{for a.e. }t\in(a,b).
\]
Consequently,
\[
\mathcal E_{p,\theta}(v_k;a,b)
=
I_k^{-p}\int_a^b B_k(t)^p A_{p,\theta}(t)\,\mathrm{d}t.
\]
For almost every $t\in(a,b)$ one has
\[
B_k(t)^p A_{p,\theta}(t)\le B_k(t).
\]
Indeed, if $B(t)\le k$, then $B_k(t)=B(t)$ and
\[
B_k(t)^p A_{p,\theta}(t)=B(t)^p A_{p,\theta}(t)=B(t)=B_k(t).
\]
If $B(t)>k$, then $B_k(t)=k$ and, since
\[
B(t)=A_{p,\theta}(t)^{-\frac{1}{p-1}}>k,
\]
it follows that
\[
A_{p,\theta}(t)<k^{-(p-1)},
\]
hence
\[
B_k(t)^p A_{p,\theta}(t)=k^p A_{p,\theta}(t)\le k=B_k(t).
\]
Therefore,
\[
\mathcal E_{p,\theta}(v_k;a,b)
\le
I_k^{-p}\int_a^b B_k(t)\,\mathrm{d}t
=
I_k^{1-p}.
\]

If
\[
I:=\int_a^b B(t)\,\mathrm{d}t<\infty,
\]
then, by monotone convergence, $I_k\uparrow I$, and therefore
\[
\operatorname{cap}_{p,\theta}(a,b)
\le
\limsup_{k\to\infty}\mathcal E_{p,\theta}(v_k;a,b)
\le
\lim_{k\to\infty} I_k^{1-p}
=
I^{1-p}.
\]
Combining this with \eqref{eq:cota-cantidad-reducida-inferior}, we obtain the explicit formula
\begin{equation}
	\label{eq:capacidad-en-terminos-resistencia}
	\operatorname{cap}_{p,\theta}(a,b)=I^{1-p}.
\end{equation}

Now, if instead
\[
\int_a^b B(t)\,\mathrm{d}t=+\infty,
\]
then $I_k\uparrow+\infty$ and
\[
\mathcal E_{p,\theta}(v_k;a,b)\le I_k^{1-p}\longrightarrow 0.
\]
Therefore,
\[
\operatorname{cap}_{p,\theta}(a,b)=0.
\]

Assume now that
\[
0<
\int_a^b A_{p,\theta}(t)^{-\frac{1}{p-1}}\,\mathrm{d}t
<\infty.
\]
Define
\begin{equation}
	\label{eq:constante-normalizacion-perfil-optimo}
	C^{-1}
	=
	\int_a^b A_{p,\theta}(t)^{-\frac{1}{p-1}}\,\mathrm{d}t.
\end{equation}
Then the profile
\begin{equation}
	\label{eq:formula-explicita-perfil-optimo}
	v_*(t)
	=
	C
	\int_a^t A_{p,\theta}(\tau)^{-\frac{1}{p-1}}\,\mathrm{d}\tau
\end{equation}
is admissible, and its derivative satisfies
\begin{equation}
	\label{eq:perfil-optimo-derivada}
	v_*'(t)=C A_{p,\theta}(t)^{-\frac{1}{p-1}}
	\qquad
	\text{for almost every }t\in(a,b).
\end{equation}
Moreover,
\[
\mathcal E_{p,\theta}(v_*;a,b)
=
C^p\int_a^b A_{p,\theta}(t)^{-\frac{1}{p-1}}\,\mathrm{d}t
=
\left(
\int_a^b A_{p,\theta}(t)^{-\frac{1}{p-1}}\,\mathrm{d}t
\right)^{1-p}.
\]
Consequently, $v_*$ is a minimizer and
\[
\operatorname{cap}_{p,\theta}(a,b)
=
\left(
\int_a^b A_{p,\theta}(t)^{-\frac{1}{p-1}}\,\mathrm{d}t
\right)^{1-p}.
\]

This concludes the proof.

We next record some additional structural consequences of the reduced problem.

	\begin{remark}[Vanishing of the weight and integrability regime]
		\label{rem:anulacion-peso-colapso-capacidad-reducida}
		The preceding formula shows that the reduced problem is determined by the integrability of
		\[
		A_{p,\theta}(t)^{-\frac{1}{p-1}}.
		\]
		In particular, if $A_{p,\theta}(t)=0$ on a subset of positive measure in $(a,b)$, then
		\[
		A_{p,\theta}(t)^{-\frac{1}{p-1}}=+\infty
		\]
		on a set of positive measure, and therefore
		\[
		\int_a^b A_{p,\theta}(t)^{-\frac{1}{p-1}}\,\mathrm{d}t=+\infty.
		\]
		Consequently,
		\[
		\operatorname{cap}_{p,\theta}(a,b)=0.
		\]
This observation will be used in the subsequent local criteria concerning degeneration of the energy weight.
	\end{remark}
	
\begin{lemma}\label{lem:aproximacion-lipschitz-problema-reducido}
	The reduced quantity does not change if one restricts the admissible class in the definition of
	\[
	\operatorname{cap}_{p,\theta}(a,b).
	\] 
	to Lipschitz functions. More precisely,
	\[
	\operatorname{cap}_{p,\theta}(a,b)
	=
	\inf\left\{
	\mathcal E_{p,\theta}(w;a,b):
	\ w\in \operatorname{Lip}([a,b]),\ 
	w(a)=0,\ 
	w(b)=1
	\right\}.
	\]
\end{lemma}

\begin{proof}
	Let
	\[
	\mathcal A_{a,b}
	:=
	\left\{
	v\in W^{1,1}(a,b):
	\ v(a)=0,\ v(b)=1
	\right\},
	\]
	and
	\[
	\mathcal A_{a,b}^{\operatorname{Lip}}
	:=
	\left\{
	w\in \operatorname{Lip}([a,b]):
	\ w(a)=0,\ w(b)=1
	\right\}.
	\]
	Since
	\[
	\mathcal A_{a,b}^{\operatorname{Lip}}\subset \mathcal A_{a,b},
	\]
	we immediately have
	\[
	\operatorname{cap}_{p,\theta}(a,b)
	\le
	\inf_{w\in \mathcal A_{a,b}^{\operatorname{Lip}}}
	\mathcal E_{p,\theta}(w;a,b).
	\]
	
	To prove the opposite inequality, we introduce
	\[
	B(t):=A_{p,\theta}(t)^{-1/(p-1)},
	\]
	with the convention that $B(t)=+\infty$ when $A_{p,\theta}(t)=0$.
	
	For each $k\ge 1$, define the truncation
	\[
	T_k(t):=\min\{B(t),k\},
	\]
	and the normalization constant
	\[
	c_k:=\int_a^b T_k(s)\,\mathrm{d}s.
	\]
	Then $0\le T_k\le k$, $T_k\in L^\infty(a,b)$, and $c_k\in(0,\infty)$ for all sufficiently large $k$. Define
	\[
	v_k(t)
	:=
	\frac{\displaystyle \int_a^t T_k(s)\,\mathrm{d}s}
	{\displaystyle c_k},
	\qquad t\in[a,b].
	\]
	By construction,
	\[
	v_k(a)=0,
	\qquad
	v_k(b)=1,
	\]
	and
	\[
	v_k'(t)=\frac{T_k(t)}{c_k}
	\quad\text{for a.e. }t\in(a,b).
	\]
	Since $0\le T_k\le k$, it follows that
	\[
	v_k\in \mathcal A_{a,b}^{\operatorname{Lip}}.
	\]
	
	Moreover, for almost every $t\in(a,b)$,
	\[
	|v_k'(t)|^p A_{p,\theta}(t)
	=
	\frac{T_k(t)^p A_{p,\theta}(t)}{c_k^p}.
	\]
	We claim that
	\[
	T_k(t)^p A_{p,\theta}(t)\le T_k(t)
	\quad\text{for a.e. }t\in(a,b).
	\]
	Indeed, if $B(t)\le k$, then $T_k(t)=B(t)=A_{p,\theta}(t)^{-1/(p-1)}$, and therefore
	\[
	T_k(t)^p A_{p,\theta}(t)
	=
	A_{p,\theta}(t)^{-p/(p-1)}A_{p,\theta}(t)
	=
	A_{p,\theta}(t)^{-1/(p-1)}
	=
	T_k(t).
	\]
	If instead $B(t)>k$, then $T_k(t)=k$ and, since
	\[
	k<B(t)=A_{p,\theta}(t)^{-1/(p-1)},
	\]
	we obtain
	\[
	k^{p-1}A_{p,\theta}(t)\le 1,
	\]
	hence
	\[
	T_k(t)^p A_{p,\theta}(t)
	=
	k^p A_{p,\theta}(t)
	\le k
	=
	T_k(t).
	\]
	Integrating, we conclude that
	\[
	\mathcal E_{p,\theta}(v_k;a,b)
	=
	\int_a^b |v_k'(t)|^p A_{p,\theta}(t)\,\mathrm{d}t
	\le
	\frac{1}{c_k^p}\int_a^b T_k(t)\,\mathrm{d}t
	=
	c_k^{\,1-p}.
	\]
	
Assume first that
	\[
	0<
	\int_a^b A_{p,\theta}(t)^{-1/(p-1)}\,\mathrm{d}t
	<
	\infty.
	\]
	Then, by monotone convergence,
	\[
	c_k
	=
	\int_a^b T_k(t)\,\mathrm{d}t,
	\qquad
	c_k \nearrow
	\int_a^b B(t)\,\mathrm{d}t
	=
	\int_a^b A_{p,\theta}(t)^{-1/(p-1)}\,\mathrm{d}t
	\quad\text{as }k\to\infty.
	\]
	Consequently,
	\[
	\limsup_{k\to\infty}\mathcal E_{p,\theta}(v_k;a,b)
	\le
	\lim_{k\to\infty} c_k^{\,1-p}
	=
	\left(
	\int_a^b A_{p,\theta}(t)^{-1/(p-1)}\,\mathrm{d}t
	\right)^{1-p}.
	\]
	By the explicit formula given in Theorem~\ref{thm:estructura-problema-reducido},
	\[
	\left(
	\int_a^b A_{p,\theta}(t)^{-1/(p-1)}\,\mathrm{d}t
	\right)^{1-p}
	=
	\operatorname{cap}_{p,\theta}(a,b).
	\]
	Thus,
	\[
	\inf_{w\in \mathcal A_{a,b}^{\operatorname{Lip}}}
	\mathcal E_{p,\theta}(w;a,b)
	\le
	\operatorname{cap}_{p,\theta}(a,b).
	\]
	
	Assume now that
	\[
	\int_a^b A_{p,\theta}(t)^{-1/(p-1)}\,\mathrm{d}t
	=
	\infty.
	\]
	Then $c_k\to\infty$, and by the preceding estimate,
	\[
	0\le \mathcal E_{p,\theta}(v_k;a,b)\le c_k^{\,1-p}\longrightarrow 0.
	\]
	Again by Theorem~\ref{thm:estructura-problema-reducido},
	\[
	\operatorname{cap}_{p,\theta}(a,b)=0.
	\]
	Therefore,
	\[
	\inf_{w\in \mathcal A_{a,b}^{\operatorname{Lip}}}
	\mathcal E_{p,\theta}(w;a,b)
	\le
	\operatorname{cap}_{p,\theta}(a,b).
	\]
	
	Combining both inequalities, we conclude that
	\[
	\operatorname{cap}_{p,\theta}(a,b)
	=
	\inf_{w\in \mathcal A_{a,b}^{\operatorname{Lip}}}
	\mathcal E_{p,\theta}(w;a,b).
	\]
\end{proof}

\subsection{Additional properties of the reduced problem}
\label{subsec:propiedades-adicionales-problema-reducido}

\begin{remark}[Coarea pushforward versus energy weight]
	\label{rem:pushforward-vs-peso-energetico}
A first reduction by coarea in the level variable suggests the geometric pushforward weight
	\begin{equation}
		\label{eq:def-wtheta-capacidad}
		w_\theta(t)
		=
		\int_{\Sigma_t\cap\Omega}
		\frac{1}{|\nabla\theta(x)|}\,\mathrm{d}\mathcal H^{n-1}(x).
	\end{equation}
	However, in the variational problem considered here the $p$-Dirichlet energy is not governed only by the geometric volume distribution, but by the effective energy weight
	\[
	A_{p,\theta}(t)
	=
	\int_{\Sigma_t\cap\Omega}
	|\nabla\theta(x)|^{p-1}\,\mathrm{d}\mathcal H^{n-1}(x).
	\]
	
	More precisely, whenever $0<w_\theta(t)<\infty$, if we define the fiberwise probability measure
	\[
	\mu_t
	:=
	\frac{|\nabla\theta|^{-1}\,\mathrm{d}\mathcal H^{n-1}}{w_\theta(t)}
	\Big|_{\Sigma_t\cap\Omega},
	\]
	then
	\begin{equation}
		\label{eq:Ap-como-w-rho}
		A_{p,\theta}(t)
		=
		w_\theta(t)\,\rho_\theta(t),
	\end{equation}
	where
	\[
	\rho_\theta(t)
	:=
	\int_{\Sigma_t\cap\Omega}
	|\nabla\theta(x)|^{p}\,\mathrm{d}\mu_t(x).
	\]
	Thus, $A_{p,\theta}$ combines the geometric pushforward $w_\theta$ with a fiberwise average of the energy factor $|\nabla\theta|^{p}$.
\end{remark}

\begin{remark}[Reduced resistance, series law, and decomposition of the weight]
	\label{rem:ley-serie-paralelo}
	We define the reduced resistance by
	\begin{equation}
		\label{eq:definicion-resistencia-reducida}
		R_{p,\theta}(a,b)
		:=
		\int_a^b A_{p,\theta}(t)^{-\frac{1}{p-1}}\,\mathrm{d}t.
	\end{equation}
	The exact formula for the reduced capacity then takes the form
	\begin{equation}
		\operatorname{cap}_{p,\theta}(a,b)
		=
		R_{p,\theta}(a,b)^{1-p}.
	\end{equation}
	Moreover, for every $a<c<b$,
	\begin{equation}
		\label{eq:ley-serie-resistencia}
		R_{p,\theta}(a,b)
		=
		R_{p,\theta}(a,c)+R_{p,\theta}(c,b).
	\end{equation}
	
	On the other hand, if for almost every $t$ the fiber $\Sigma_t\cap\Omega$ decomposes as a disjoint union of measurable components
	\[
	\Sigma_t\cap\Omega=\bigcup_{j=1}^k \Sigma_t^j,
	\]
	then
	\begin{equation}
		\label{eq:suma-componentes-peso}
		A_{p,\theta}(t)
		=
		\sum_{j=1}^k
		\int_{\Sigma_t^j}
		|\nabla\theta(x)|^{p-1}\,\mathrm{d}\mathcal H^{n-1}(x).
	\end{equation}
	Thus, the conductive contributions of the different components of the fiber add at the level of the energy weight $A_{p,\theta}$, while the one-dimensional variable $R_{p,\theta}$ composes in series by additivity of the integral.
\end{remark}

\begin{corollary}[Continuity and monotonicity with respect to the levels]
	\label{cor:continuidad-monotonicidad-niveles}
	Suppose that
	\[
	A_{p,\theta}(t)>0
	\qquad
	\text{for a.e. } t\in I,
	\]
	and that
	\[
	A_{p,\theta}^{-\frac{1}{p-1}}\in L^1_{\mathrm{loc}}(I),
	\]
	where $I\subset\mathbb R$ is an open interval. Then the map
	\[
	(a,b)\longmapsto R_{p,\theta}(a,b),
	\]
	defined by \eqref{eq:definicion-resistencia-reducida}, is continuous on $\{(a,b)\in I\times I:\ a<b\}$, and therefore so is
	\[
	(a,b)\longmapsto \operatorname{cap}_{p,\theta}(a,b).
	\]
	Moreover:
	\begin{itemize}
		\item if $a<b_1<b_2$, then
		\[
		R_{p,\theta}(a,b_2)\ge R_{p,\theta}(a,b_1),
		\qquad
		\operatorname{cap}_{p,\theta}(a,b_2)
		\le
		\operatorname{cap}_{p,\theta}(a,b_1);
		\]
		\item if $a_1<a_2<b$, then
		\[
		R_{p,\theta}(a_1,b)\ge R_{p,\theta}(a_2,b),
		\qquad
		\operatorname{cap}_{p,\theta}(a_1,b)
		\le
		\operatorname{cap}_{p,\theta}(a_2,b).
		\]
	\end{itemize}
	If in addition $A_{p,\theta}(t)<\infty$ for a.e. $t\in I$, then both inequalities are strict.
\end{corollary}

\begin{proof}
	By definition,
	\[
	R_{p,\theta}(a,b)
	=
	\int_a^b A_{p,\theta}(t)^{-\frac{1}{p-1}}\,\mathrm{d}t.
	\]
	Since $A_{p,\theta}^{-\frac{1}{p-1}}\in L^1_{\mathrm{loc}}(I)$, the map
	\[
	(a,b)\longmapsto R_{p,\theta}(a,b)
	\]
	is continuous on the domain $a<b$. The continuity of $(a,b)\mapsto \operatorname{cap}_{p,\theta}(a,b)$ then follows from \eqref{eq:capacidad-en-terminos-resistencia}.
	
	If $a<b_1<b_2$, the additivity \eqref{eq:ley-serie-resistencia} gives
	\[
	R_{p,\theta}(a,b_2)
	=
	R_{p,\theta}(a,b_1)+R_{p,\theta}(b_1,b_2)
	\ge
	R_{p,\theta}(a,b_1),
	\]
	and since $s\mapsto s^{1-p}$ is decreasing on $(0,\infty)$, we obtain
	\[
	\operatorname{cap}_{p,\theta}(a,b_2)
	\le
	\operatorname{cap}_{p,\theta}(a,b_1).
	\]
	Similarly, if $a_1<a_2<b$, then
	\[
	R_{p,\theta}(a_1,b)
	=
	R_{p,\theta}(a_1,a_2)+R_{p,\theta}(a_2,b)
	\ge
	R_{p,\theta}(a_2,b),
	\]
	and therefore
	\[
	\operatorname{cap}_{p,\theta}(a_1,b)
	\le
	\operatorname{cap}_{p,\theta}(a_2,b).
	\]
	If in addition $A_{p,\theta}(t)<\infty$ for a.e. $t\in I$, then
	\[
	A_{p,\theta}(t)^{-\frac{1}{p-1}}>0
	\qquad
	\text{for a.e. }t\in I,
	\]
	and every strict enlargement of the interval adds a strictly positive contribution to $R_{p,\theta}$. Hence the strict monotonicity of $R_{p,\theta}$, and consequently the opposite strict monotonicity of $\operatorname{cap}_{p,\theta}$.
\end{proof}

\begin{lemma}[Invariance under increasing reparametrization of the phase]
	\label{lem:invariancia-reparametrizacion-fase}
	Let $\phi:\mathbb R\to\mathbb R$ be a Lipschitz and strictly increasing function, and define
	\[
	\widetilde\theta:=\phi\circ\theta.
	\]
Then, for almost every $t\in(a,b)$ at which $\phi'(t)$ exists, and writing $s=\phi(t)$, one has
	\begin{equation}
		\label{eq:transformacion-peso-reparametrizacion}
		A_{p,\widetilde\theta}(s)
		=
		\phi'(t)^{p-1}A_{p,\theta}(t).
	\end{equation}
	Consequently,
	\begin{equation}
		\label{eq:invariancia-resistencia-reducida}
		R_{p,\widetilde\theta}(\phi(a),\phi(b))
		=
		R_{p,\theta}(a,b),
	\end{equation}
	and therefore
	\begin{equation}
		\label{eq:invariancia-capacidad-reducida}
		\operatorname{cap}_{p,\widetilde\theta}(\phi(a),\phi(b))
		=
		\operatorname{cap}_{p,\theta}(a,b).
	\end{equation}
\end{lemma}

\begin{proof}
	Since $\phi$ is Lipschitz, it is differentiable almost everywhere, and by the chain rule
	\[
	\nabla(\phi\circ\theta)(x)
	=
	\phi'(\theta(x))\,\nabla\theta(x)
	\qquad
	\text{for a.e. }x\in\Omega.
	\]
	Moreover, since $\phi$ is strictly increasing, for each $t\in(a,b)$ and $s=\phi(t)$ one has
	\[
	\{\widetilde\theta=s\}=\{\theta=t\}.
	\]
	Therefore, for almost every $t\in(a,b)$,
	\[
	A_{p,\widetilde\theta}(\phi(t))
	=
	\int_{\Sigma_t\cap\Omega}
	|\phi'(\theta(x))|^{p-1}|\nabla\theta(x)|^{p-1}\,\mathrm{d}\mathcal H^{n-1}(x)
	=
	\phi'(t)^{p-1}A_{p,\theta}(t),
	\]
	which proves \eqref{eq:transformacion-peso-reparametrizacion}.
	
	Consequently,
	\[
	A_{p,\widetilde\theta}(\phi(t))^{-\frac{1}{p-1}}\,\phi'(t)
	=
	A_{p,\theta}(t)^{-\frac{1}{p-1}}
	\qquad
	\text{for a.e. }t\in(a,b).
	\]
	Integrating and using the change of variables $s=\phi(t)$, we obtain
	\[
	R_{p,\widetilde\theta}(\phi(a),\phi(b))
	=
	R_{p,\theta}(a,b).
	\]
	The identity \eqref{eq:invariancia-capacidad-reducida} then follows from \eqref{eq:capacidad-en-terminos-resistencia}.
\end{proof}

\begin{remark}[On decreasing reparametrizations]
	\label{rem:reparametrizacion-decreciente}
	The preceding invariance is formulated for Lipschitz strictly increasing reparametrizations, which preserve the orientation of the level interval $(a,b)$. If $\phi$ is strictly decreasing, then it interchanges the roles of the endpoints and the corresponding formulation is written on the interval $(\phi(b),\phi(a))$. We will not use this variant in what follows.
\end{remark}

\begin{proposition}[Improvement over the linear profile]
	\label{prop:mejora-perfil-lineal}
	Let $\theta:\Omega\to\mathbb{R}$ be a phase in $W^{1,\infty}(\Omega)$, and let $a<b$ be admissible levels such that
	\[
	\overline{E_a}\subset\Omega,
	\qquad
	\overline{F_b}\subset\Omega,
	\qquad
	\overline{E_a}\cap\overline{F_b}=\varnothing.
	\]
	Assume in addition that
	\[
	A_{p,\theta}(t)>0
	\qquad
	\text{for a.e. }t\in(a,b),
	\]
	and that
	\[
	A_{p,\theta}^{-1/(p-1)}\in L^1(a,b).
	\]
	Then
	\begin{equation}
		\label{eq:mejora-perfil-lineal}
		\operatorname{cap}_{p,\theta}(a,b)
		\le
		\frac{1}{(b-a)^p}\int_a^b A_{p,\theta}(t)\,\mathrm{d}t.
	\end{equation}
	More precisely, if
	\[
	v_{\mathrm{lin}}(t):=\frac{t-a}{b-a},
	\qquad t\in[a,b],
	\]
	then
	\begin{equation}
		\label{eq:energia-perfil-lineal}
		\mathcal E_{p,\theta}(v_{\mathrm{lin}};a,b)
		=
		\frac{1}{(b-a)^p}\int_a^b A_{p,\theta}(t)\,\mathrm{d}t,
	\end{equation}
	and
	\begin{equation}
		\label{eq:capacidad-reducida-mejor-lineal}
		\operatorname{cap}_{p,\theta}(a,b)
		\le
		\mathcal E_{p,\theta}(v_{\mathrm{lin}};a,b).
	\end{equation}
	Moreover, equality in \eqref{eq:capacidad-reducida-mejor-lineal} holds if and only if
	\[
	A_{p,\theta}(t)
	\quad
	\text{is constant for a.e. } t\in(a,b).
	\]
\end{proposition}

\begin{proof}
	Consider the linear profile
	\[
	v_{\mathrm{lin}}(t):=\frac{t-a}{b-a},
	\qquad t\in[a,b].
	\]
	Then
	\[
	v'_{\mathrm{lin}}(t)=\frac{1}{b-a}
	\qquad
	\text{for a.e. }t\in(a,b),
	\]
	and by the definition of the reduced functional,
	\[
	\mathcal E_{p,\theta}(v_{\mathrm{lin}};a,b)
	=
	\frac{1}{(b-a)^p}\int_a^b A_{p,\theta}(t)\,\mathrm{d}t.
	\]
	This proves \eqref{eq:energia-perfil-lineal}.
	
	Since $\operatorname{cap}_{p,\theta}(a,b)$ is the infimum of $\mathcal E_{p,\theta}(v;a,b)$ over all admissible profiles, it follows immediately that
	\[
	\operatorname{cap}_{p,\theta}(a,b)
	\le
	\mathcal E_{p,\theta}(v_{\mathrm{lin}};a,b)
	=
	\frac{1}{(b-a)^p}\int_a^b A_{p,\theta}(t)\,\mathrm{d}t,
	\]
	which proves \eqref{eq:capacidad-reducida-mejor-lineal}.
	
	On the other hand, using the exact formula
	\[
	\operatorname{cap}_{p,\theta}(a,b)
	=
	\left(
	\int_a^b A_{p,\theta}(t)^{-1/(p-1)}\,\mathrm{d}t
	\right)^{1-p},
	\]
	the preceding inequality is equivalent to
	\[
	(b-a)^p
	\le
	\left(\int_a^b A_{p,\theta}(t)\,\mathrm{d}t\right)
	\left(
	\int_a^b A_{p,\theta}(t)^{-1/(p-1)}\,\mathrm{d}t
	\right)^{p-1},
	\]
	which is precisely Hölder's inequality applied to the functions
	\[
	A_{p,\theta}(t)^{1/p}
	\quad\text{and}\quad
	A_{p,\theta}(t)^{-1/p}
	\]
	with conjugate exponents $p$ and $p/(p-1)$.
	
	Moreover, the equality condition in Hölder is equivalent to requiring that these two functions be proportional almost everywhere, which happens if and only if $A_{p,\theta}(t)$ is constant for almost every $t\in(a,b)$. This characterizes equality in \eqref{eq:capacidad-reducida-mejor-lineal}.
\end{proof}

\begin{remark}
	\label{rem:perfil-lineal-como-control}
	The preceding proposition records an elementary upper bound within the reduced problem: the linear profile provides an explicit admissible profile, while $\operatorname{cap}_{p,\theta}(a,b)$ represents the optimal cost in the fibered class. In particular,
	\[
	\mathcal E_{p,\theta}(v_{\mathrm{lin}};a,b)-\operatorname{cap}_{p,\theta}(a,b)
	\]
	measures the excess cost of the linear profile relative to the reduced optimum. In the explicit models of Section~\ref{sec:ejemplos}, this comparison appears concretely.
\end{remark}
\section{Geometric comparability and proof of Theorem~\ref{thm:cota-superior-fibrada}}
\label{sec:comparabilidad-geometrica}

In this section we establish the upper bound obtained by fibered restriction and derive estimates for the energy weight $A_{p,\theta}$ from the gradient profile and the geometry of the level fibers. 

\subsection{Upper bound by fibered restriction}
\label{subsec:cota-superior-restriccion-fibrada}

Let $\theta:\overline\Omega\to\mathbb R$ be a Lipschitz phase on $\overline\Omega$, and let $a<b$ be admissible levels for the phase such that the pair $(E_a,F_b)$ forms an interior condenser in $\Omega$.

The fibered class associated with $\theta$ yields, without additional hypotheses beyond the framework already fixed, an upper bound for the full geometric capacity of the condenser induced by the phase, since Lipschitz admissible profiles for the reduced problem generate admissible functions for the geometric capacity by composition with $\theta$.

\paragraph{Proof of Theorem~\ref{thm:cota-superior-fibrada}.}
	Let $\varepsilon>0$. By Lemma~\ref{lem:aproximacion-lipschitz-problema-reducido}, for every $\varepsilon>0$ there exists a Lipschitz admissible profile $v_\varepsilon$ for the reduced problem whose energy is $\varepsilon$-almost optimal.
	\[
	v_\varepsilon:[a,b]\to\mathbb R,
	\qquad
	v_\varepsilon(a)=0,
	\quad
	v_\varepsilon(b)=1,
	\]
	such that
	\[
	\mathcal E_{p,\theta}(v_\varepsilon;a,b)
	\le
	\operatorname{cap}_{p,\theta}(a,b)+\varepsilon.
	\]
	Let $\widetilde v_\varepsilon$ be the truncated extension of $v_\varepsilon$ given by Lemma~\ref{lem:extension-truncada-composicion}, and define
	\[
	u_\varepsilon(x):=\widetilde v_\varepsilon(\theta(x)).
	\]
	Then
	\[
	u_\varepsilon=0 \quad \text{on } E_a,
	\qquad
	u_\varepsilon=1 \quad \text{on } F_b,
	\]
	so that $u_\varepsilon\in \mathcal A_{E_a,F_b}(\Omega)$. Moreover, by Proposition~\ref{prop:identidad-rigurosa-energia-fibrada},
	\[
	\int_\Omega |\nabla u_\varepsilon(x)|^p\,\mathrm{d}x
	=
	\mathcal E_{p,\theta}(v_\varepsilon;a,b)
	\le
	\operatorname{cap}_{p,\theta}(a,b)+\varepsilon.
	\]
Letting $\varepsilon\downarrow 0$, we obtain
\[
\operatorname{Cap}_{p,\Omega}(E_a,F_b)\le \operatorname{cap}_{p,\theta}(a,b),
\]
as claimed.

\subsection{Estimates for the energy weight and the reduced capacity}
\label{subsec:diseno-peso-cantidad-reducida}

Once the phase $\theta$ is fixed, control of the reduced quantity $\operatorname{cap}_{p,\theta}(a,b)$
is obtained by estimating the energy weight
\[
A_{p,\theta}(t)
=
\int_{\Sigma_t\cap\Omega}
|\nabla\theta(x)|^{p-1}\,\mathrm{d}\mathcal H^{n-1}(x).
\]
Using the notation
\[
S_\theta(t)=\mathcal H^{n-1}(\Sigma_t\cap\Omega)
\]
introduced in Section~\ref{sec:setup-capacidad-fase}, the energy weight is determined jointly by the size of the gradient on each fiber and by the geometric size of the levels. In particular, control of both profiles as functions of $t$ yields explicit bounds for $A_{p,\theta}$ and, through the exact formula for the reduced problem, for $\operatorname{cap}_{p,\theta}(a,b)$.

\begin{proposition}[Two-sided design of the weight and the reduced capacity]
	\label{prop:diseno-bilateral-peso-cantidad}
	Suppose that there exist measurable functions
	\[
	\Gamma_-,\Gamma_+:[a,b]\to [0,\infty),
	\qquad
	m,M:[a,b]\to [0,\infty),
	\]
	such that
	\[
	\Gamma_-(\theta(x))
	\le
	|\nabla\theta(x)|
	\le
	\Gamma_+(\theta(x))
	\qquad
	\text{for a.e. } x\in \{a\le \theta\le b\},
	\]
	and
	\[
	m(t)\le S_\theta(t)\le M(t)
	\qquad
	\text{for a.e. } t\in (a,b).
	\]
	Then, for almost every $t\in (a,b)$,
	\begin{equation}
		\label{eq:diseno-bilateral-peso}
		\Gamma_-(t)^{p-1}m(t)
		\le
		A_{p,\theta}(t)
		\le
		\Gamma_+(t)^{p-1}M(t).
	\end{equation}
	If in addition the integrals involved are finite, one has
	\begin{equation}
		\label{eq:diseno-bilateral-capacidad-inferior}
		\left(
		\int_a^b
		\left(\Gamma_-(t)^{p-1}m(t)\right)^{-\frac{1}{p-1}}
		\,\mathrm{d}t
		\right)^{1-p}
		\le
		\operatorname{cap}_{p,\theta}(a,b),
	\end{equation}
	and
	\begin{equation}
		\label{eq:diseno-bilateral-capacidad-superior}
		\operatorname{cap}_{p,\theta}(a,b)
		\le
		\left(
		\int_a^b
		\left(\Gamma_+(t)^{p-1}M(t)\right)^{-\frac{1}{p-1}}
		\,\mathrm{d}t
		\right)^{1-p}.
	\end{equation}
\end{proposition}

\begin{proof}
	Fixing $t\in (a,b)$, the hypothesis on the gradient implies that, for almost every $x\in \Sigma_t\cap\Omega$,
	\[
	\Gamma_-(t)\le |\nabla\theta(x)|\le \Gamma_+(t).
	\]
	Raising to the power $p-1$ and integrating over the fiber $\Sigma_t\cap\Omega$, we obtain
	\[
	\Gamma_-(t)^{p-1}S_\theta(t)
	\le
	A_{p,\theta}(t)
	\le
	\Gamma_+(t)^{p-1}S_\theta(t).
	\]
	Using in addition that
	\[
	m(t)\le S_\theta(t)\le M(t),
	\]
	we conclude \eqref{eq:diseno-bilateral-peso}.
	
By the exact formula for the reduced capacity,
\[
\operatorname{cap}_{p,\theta}(a,b)
=
\left(
\int_a^b A_{p,\theta}(t)^{-\frac{1}{p-1}}\,\mathrm{d}t
\right)^{1-p}.
\]
	From \eqref{eq:diseno-bilateral-peso} it follows that
	\[
	\left(\Gamma_+(t)^{p-1}M(t)\right)^{-\frac{1}{p-1}}
	\le
	A_{p,\theta}(t)^{-\frac{1}{p-1}}
	\le
	\left(\Gamma_-(t)^{p-1}m(t)\right)^{-\frac{1}{p-1}}.
	\]
	Integrating over $(a,b)$, we obtain
	\[
	\int_a^b
	\left(\Gamma_+(t)^{p-1}M(t)\right)^{-\frac{1}{p-1}}
	\,\mathrm{d}t
	\le
	\int_a^b
	A_{p,\theta}(t)^{-\frac{1}{p-1}}
	\,\mathrm{d}t
	\le
	\int_a^b
	\left(\Gamma_-(t)^{p-1}m(t)\right)^{-\frac{1}{p-1}}
	\,\mathrm{d}t.
	\]
	Since the function $s\mapsto s^{1-p}$ is decreasing on $(0,\infty)$, this yields \eqref{eq:diseno-bilateral-capacidad-inferior} and \eqref{eq:diseno-bilateral-capacidad-superior}.
\end{proof}

In particular, $\operatorname{cap}_{p,\theta}(a,b)$ is controlled by the gradient of the phase and by the geometric size of the fibers.

\begin{corollary}[Exact design regime]
	\label{cor:regimen-exacto-diseno}
	Suppose that there exists a measurable function
	\[
	\Gamma:(a,b)\to [0,\infty)
	\]
	such that
	\[
	|\nabla\theta(x)|=\Gamma(\theta(x))
	\qquad
	\text{for a.e. } x\in \{a<\theta<b\}.
	\]
	Then, for almost every $t\in (a,b)$,
	\begin{equation}
		\label{eq:regimen-exacto-peso}
		A_{p,\theta}(t)=\Gamma(t)^{p-1}S_\theta(t).
	\end{equation}
	Consequently,
	\begin{equation}
		\label{eq:regimen-exacto-cantidad-reducida}
		\operatorname{cap}_{p,\theta}(a,b)
		=
		\left(
		\int_a^b
		\Gamma(t)^{-1}S_\theta(t)^{-\frac{1}{p-1}}
		\,\mathrm{d}t
		\right)^{1-p},
	\end{equation}
	whenever the integral is finite.
\end{corollary}

\begin{proof}
	It is enough to apply Proposition~\ref{prop:diseno-bilateral-peso-cantidad} with
	\[
	\Gamma_-=\Gamma_+=\Gamma,
	\qquad
	m=M=S_\theta.
	\]
	The identity \eqref{eq:regimen-exacto-peso} follows directly from the definition of $A_{p,\theta}$, and substituting it into the exact formula for the reduced capacity yields \eqref{eq:regimen-exacto-cantidad-reducida}.
\end{proof}

In this regime, the reduced capacity is determined by the gradient law $\Gamma$ and by the size of the fiber.
\section{Critical regimes of the phase and proof of Theorem~\ref{thm:umbral-local-transmisibilidad}}
\label{sec:regimenes-criticos-theta}

In this section we study the local behaviour of the reduced problem near a critical level of the phase $\theta$. The analysis is formulated through a localized energy weight and leads to an explicit threshold of transmissibility in terms of the degeneration of the gradient and the geometric collapse of the fibers.

\subsection{Local weight and critical mechanism}
\label{subsec:peso-local-mecanismo-critico}

Let $\theta:\Omega\to\mathbb R$ be a phase in $W^{1,\infty}(\Omega)$. We fix a point $x_0\in\Omega$ such that
\[
\nabla\theta(x_0)=0,
\qquad
t_0:=\theta(x_0),
\]
and choose an open set $U\subset\Omega$ with $x_0\in U$.

\begin{definition}[Critical local weight]
	\label{def:peso-local-critico}
	For almost every $t$ near $t_0$, we define
	\[
	A_{p,\theta}^{U}(t)
	:=
	\int_{\Sigma_t\cap U}
	|\nabla\theta(x)|^{p-1}\,\mathrm{d}\mathcal H^{n-1}(x).
	\]
\end{definition}

We also use the notation
\[
S_\theta^{U}(t):=\mathcal H^{n-1}(\Sigma_t\cap U)
\]
for the localized size of the fiber.

The quantity $A_{p,\theta}^{U}(t)$ represents the energy contribution of the region $U$ at level $t$, while $S_\theta^{U}(t)$ records the corresponding local size of the fiber. Localization is essential, since a global fiber $\Sigma_t$ may contain components far from $x_0$ whose behaviour does not reflect the critical regime under consideration.

\begin{remark}[Local weight versus global weight]
	\label{rem:peso-local-vs-global}
	The global weight
	\[
	A_{p,\theta}(t)
	=
	\int_{\Sigma_t\cap\Omega}
	|\nabla\theta(x)|^{p-1}\,\mathrm{d}\mathcal H^{n-1}(x)
	\]
may remain uniformly positive even if there exists a local critical region where the gradient degenerates. For this reason, the analysis of the effect of a critical level on the reduced problem is first formulated through a localized weight such as $A_{p,\theta}^{U}$.
\end{remark}

The local structure of the weight depends on two ingredients:

\begin{itemize}
	\item the degeneration of the gradient on fibers near the critical level;
	\item the geometric variation of the size of the localized fibers.
\end{itemize}

The degeneration of the gradient can be controlled, for example, by means of a \L ojasiewicz-type inequality; see \cite{Chill2003}. If $\theta$ is analytic in a neighbourhood of $x_0$, then there exist an open set $U$ with $x_0\in U$, a constant $C_0>0$, and an exponent $\alpha\in[0,1)$ such that
\begin{equation}
	\label{eq:lojasiewicz-gradiente}
	|\nabla\theta(x)|
	\ge
	C_0\,|\theta(x)-t_0|^\alpha
	\qquad
	\text{for all } x\in U.
\end{equation}
In particular, for $t$ sufficiently close to $t_0$, on the localized fiber $\Sigma_t\cap U$ one obtains
\begin{equation}
	\label{eq:lojasiewicz-fibra-local}
	|\nabla\theta(x)|
	\ge
	C_0\,|t-t_0|^\alpha
	\qquad
	\text{for all } x\in \Sigma_t\cap U.
\end{equation}

The geometric collapse is encoded in the measure of the localized fiber. If there exists an exponent $\nu\ge 0$ such that
\begin{equation}
	\label{eq:fibra-colapso}
	S_\theta^{U}(t)
	\asymp
	|t-t_0|^\nu
	\qquad
	\text{as } t\to t_0^+,
\end{equation}
then the size of the fiber collapses with order $\nu$ as one approaches the critical level.

From \eqref{eq:lojasiewicz-fibra-local} and \eqref{eq:fibra-colapso} one obtains an explicit lower bound for the local weight.

\begin{lemma}[Local lower bound for the critical weight]
	\label{lem:cota-local-inferior-peso-critico}
	Suppose that \eqref{eq:lojasiewicz-gradiente} holds in $U$ with exponent $\alpha\in[0,1)$, and that moreover
	\[
	S_\theta^{U}(t)
	\asymp
	|t-t_0|^\nu
	\qquad
	\text{as } t\to t_0^+,
	\]
	for some $\nu\ge 0$. Then
	\begin{equation}
		\label{eq:cota-local-peso-critico}
		A_{p,\theta}^{U}(t)
		\gtrsim
		|t-t_0|^{\alpha(p-1)+\nu}
		\qquad
		\text{as } t\to t_0^+.
	\end{equation}
\end{lemma}

\begin{proof}
	By \eqref{eq:lojasiewicz-fibra-local}, for every $x\in \Sigma_t\cap U$ one has
	\[
	|\nabla\theta(x)|^{p-1}
	\ge
	C_0^{p-1}|t-t_0|^{\alpha(p-1)}.
	\]
	Integrating over $\Sigma_t\cap U$, we obtain
	\[
	A_{p,\theta}^{U}(t)
	\ge
	C_0^{p-1}|t-t_0|^{\alpha(p-1)}
	S_\theta^{U}(t).
	\]
	Using now \eqref{eq:fibra-colapso}, we conclude \eqref{eq:cota-local-peso-critico}.
\end{proof}

The preceding lemma provides a local lower bound for the critical weight. We next introduce a two-sided hypothesis on that profile in order to obtain a criterion for local integrability of the reduced resistance.

\begin{remark}[On the two-sidedness of the profile]
	\label{rem:bilateralidad-perfil-peso-critico}
	The lower estimate \eqref{eq:cota-local-peso-critico} is a direct consequence of the \L ojasiewicz-type control and of the geometric profile of the fiber. By contrast, a two-sided equivalence
	\begin{equation}
		\label{eq:perfil-peso-local-bilateral}
		A_{p,\theta}^{U}(t)
		\asymp
		|t-t_0|^{\alpha(p-1)+\nu}
		\qquad
		\text{as } t\to t_0^+
	\end{equation}
	requires additional hypotheses: one needs an upper control of the gradient, or a finer local description of the geometry of $\theta$ near the critical level.
\end{remark}

\subsection{Local resistance and transmissibility criterion}
\label{subsec:resistencia-local-criterio-transmisibilidad}

The exact formula for the reduced problem leads to the following effective resistance associated with the local weight.

\begin{definition}[Local resistance]
	\label{def:resistencia-local}
	Let $\delta>0$ be such that $(t_0,t_0+\delta)$ lies within the local regime under consideration. We define
	\[
	R_{p,\theta}^{U}(t_0,\delta)
	:=
	\int_{t_0}^{t_0+\delta}
	\left(A_{p,\theta}^{U}(t)\right)^{-\frac{1}{p-1}}\,\mathrm{d}t.
	\]
\end{definition}

The interpretation is local and refers only to the reduced model associated with the localized weight $A_{p,\theta}^{U}$. In particular, $R_{p,\theta}^{U}(t_0,\delta)$ measures the local integrability of the fibered cost near the critical level. Under an appropriate control of the weight profile, its finiteness or divergence is determined by an explicit threshold.

\begin{lemma}[Local threshold for reduced transmissibility]
	\label{lem:umbral-local-transmisibilidad-reducida}
	Suppose that there exist $\alpha\in[0,1)$ and $\nu\ge 0$ such that
	\begin{equation}
		\label{eq:hipotesis-umbral-local-bilateral}
		A_{p,\theta}^{U}(t)
		\asymp
		|t-t_0|^{\alpha(p-1)+\nu}
		\qquad
		\text{as } t\to t_0^+.
	\end{equation}
	Then
	\begin{equation}
		\label{eq:criterio-umbral-local}
		R_{p,\theta}^{U}(t_0,\delta)<\infty
		\quad\Longleftrightarrow\quad
		\alpha+\frac{\nu}{p-1}<1.
	\end{equation}
	Consequently, under the two-sided hypothesis \eqref{eq:hipotesis-umbral-local-bilateral}, the reduced resistance is integrable on sufficiently small one-sided intervals to the right of $t_0$ if and only if
	\begin{equation}
		\label{eq:transmisibilidad-local-criterio}
		\alpha+\frac{\nu}{p-1}<1.
	\end{equation}
\end{lemma}

\begin{proof}
	Under \eqref{eq:hipotesis-umbral-local-bilateral}, one has
	\[
	\left(A_{p,\theta}^{U}(t)\right)^{-\frac{1}{p-1}}
	\asymp
	|t-t_0|^{-\alpha-\frac{\nu}{p-1}}
	\qquad
	\text{as } t\to t_0^+.
	\]
	Therefore, the finiteness of $R_{p,\theta}^{U}(t_0,\delta)$ is equivalent to the convergence of
	\[
	\int_{t_0}^{t_0+\delta}
	|t-t_0|^{-\alpha-\frac{\nu}{p-1}}\,\mathrm{d}t.
	\]
	This integral converges if and only if
	\[
	\alpha+\frac{\nu}{p-1}<1.
	\]
	This proves \eqref{eq:criterio-umbral-local} and, consequently, \eqref{eq:transmisibilidad-local-criterio}.
\end{proof}

\begin{remark}[Scope of the local criterion]
	\label{rem:alcance-criterio-local}
	The local criterion is formulated at the level of the reduced model.
	In the critical or supercritical regime, once the vanishing of the reduced quantity has been established, the upper bound by fibered restriction implies the corresponding vanishing of the full geometric capacity.
	Outside that regime, comparability or exactness requires an additional comparison step between the fibered problem and the full variational problem.
\end{remark}

The preceding results describe the critical threshold at the level of the localized weight $A_{p,\theta}^{U}$. We next introduce a two-sided hypothesis directly on the global energy weight $A_{p,\theta}$ near $t_0$, in order to obtain a consequence for the global reduced quantity associated with the level interval $(t_0,t_0+\delta)$.

\begin{lemma}[Critical or supercritical local regime and vanishing of the reduced quantity]
	\label{lem:regimen-supercritico-capacidad-cero}
	Let $t_0\in\mathbb{R}$ and suppose that there exist $\delta_0>0$, constants $c_1,c_2>0$, and parameters $\alpha\in[0,1)$, $\nu\ge0$ such that
	\begin{equation}
		\label{eq:perfil-supercritico-local}
		c_1 (t-t_0)^{\alpha(p-1)+\nu}
		\le
		A_{p,\theta}(t)
		\le
		c_2 (t-t_0)^{\alpha(p-1)+\nu}
	\end{equation}
	for almost every $t\in(t_0,t_0+\delta_0)$. If
	\begin{equation}
		\label{eq:umbral-supercritico-local}
		\alpha+\frac{\nu}{p-1}\ge 1,
	\end{equation}
	then, for every $\delta\in(0,\delta_0)$,
	\begin{equation}
		\label{eq:capacidad-reducida-cero-supercritico}
		\operatorname{cap}_{p,\theta}(t_0,t_0+\delta)=0.
	\end{equation}
	Consequently, whenever the sets induced by the phase form an interior condenser,
	\begin{equation}
		\label{eq:capacidad-geometrica-cero-supercritico}
		\operatorname{Cap}_{p,\Omega}(E_{t_0},F_{t_0+\delta})=0.
	\end{equation}
\end{lemma}

\begin{proof}
	By \eqref{eq:perfil-supercritico-local},
	\[
	A_{p,\theta}(t)^{-1/(p-1)}
	\asymp
	(t-t_0)^{-\alpha-\nu/(p-1)}
	\qquad
	\text{as } t\downarrow t_0.
	\]
	If \eqref{eq:umbral-supercritico-local} holds, then
	\[
	\int_{t_0}^{t_0+\delta} A_{p,\theta}(t)^{-1/(p-1)}\,\mathrm{d}t
	=
	+\infty.
	\]
The exact formula for the reduced problem therefore implies \eqref{eq:capacidad-reducida-cero-supercritico}. By the upper bound by fibered restriction,
\[
\operatorname{Cap}_{p,\Omega}(E_{t_0},F_{t_0+\delta})
\le
\operatorname{cap}_{p,\theta}(t_0,t_0+\delta)
=
0,
\]
and the nonnegativity of the geometric capacity yields \eqref{eq:capacidad-geometrica-cero-supercritico}.
\end{proof}

\paragraph{Proof of Theorem~\ref{thm:umbral-local-transmisibilidad}.}
Under the hypothesis of the theorem, there exist constants $c_1,c_2>0$ such that
\[
c_1 (t-t_0)^{\alpha(p-1)+\nu}
\le
A_{p,\theta}(t)
\le
c_2 (t-t_0)^{\alpha(p-1)+\nu}
\]
for almost every $t\in(t_0,t_0+\delta_0)$. Consequently,
\[
A_{p,\theta}(t)^{-1/(p-1)}
\asymp
(t-t_0)^{-\alpha-\nu/(p-1)}
\qquad
\text{for } t\downarrow t_0.
\]
Therefore, for every $\delta\in(0,\delta_0)$, the integral
\[
\int_{t_0}^{t_0+\delta} A_{p,\theta}(t)^{-1/(p-1)}\,\mathrm{d}t
\]
is finite if and only if
\[
\alpha+\frac{\nu}{p-1}<1.
\]
This proves the first assertion of the theorem.

If in addition
\[
\alpha+\frac{\nu}{p-1}\ge 1,
\]
then Lemma~\ref{lem:regimen-supercritico-capacidad-cero} implies that
\[
\operatorname{cap}_{p,\theta}(t_0,t_0+\delta)=0
\qquad
\text{for every }\delta\in(0,\delta_0).
\]
This concludes the proof of Theorem~\ref{thm:umbral-local-transmisibilidad}.

\begin{remark}
	\label{rem:supercritico-como-criterio-suficiente}
	The preceding result identifies a sufficient criterion for the vanishing of the reduced quantity, determined by the supercritical regime of the effective energy weight. In particular, the critical condition is reflected in the integrability of
	\[
	A_{p,\theta}(t)^{-1/(p-1)}
	\]
	near the level under consideration.
\end{remark}
\section{Examples and explicit models}
\label{sec:ejemplos}

In this section we collect examples in which the reduced capacity $
\operatorname{cap}_{p,\theta}(a,b)$ can be computed explicitly and in which the energy weight $
A_{p,\theta}(t)$ admits a direct geometric interpretation. The models considered isolate the mechanisms identified in the previous sections.

We distinguish two aspects. On the one hand, the reduced problem depends only on the interval of levels crossed and on the weight $A_{p,\theta}$. On the other hand, in a bounded domain, the sets
\[
E_a=\{\theta\le a\},
\qquad
F_b=\{\theta\ge b\}
\]
do not always form an interior condenser, because the upper plate may touch $\partial\Omega$. Consequently, some of the following examples are interpreted primarily as exact models of the reduced problem and as explicit configurations compatible with the same phase.

\subsection{The planar model}
\label{subsec:modelo-plano}

Let
\[
\Omega=(-L,L)\times D\subset\mathbb{R}^n,
\]
where $L>0$ and $D\subset\mathbb{R}^{n-1}$ is a bounded open set. Consider the phase
\[
\theta(x)=x_1.
\]
Then, for $t\in(-L,L)$,
\[
\Sigma_t=\{x\in\Omega:x_1=t\}=\{t\}\times D,
\qquad
|\nabla\theta|=1.
\]
Consequently, the fiber size is constant:
\[
S_\theta(t)=\mathcal H^{n-1}(D).
\]
Since moreover $|\nabla\theta|^{p-1}=1$, one obtains
\begin{equation}
	\label{eq:peso-plano}
	A_{p,\theta}(t)=\mathcal H^{n-1}(D)
	\qquad
	\text{for every } t\in(-L,L).
\end{equation}

Therefore, for every $-L<a<b<L$,
\begin{equation}
	\label{eq:funcional-plano}
	\mathcal{E}_{p,\theta}(v;a,b)
	=
	\mathcal H^{n-1}(D)\int_a^b |v'(t)|^p\,\mathrm{d}t,
\end{equation}
and the exact formula for the reduced problem yields
\begin{equation}
	\label{eq:capacidad-reducida-plano}
	\operatorname{cap}_{p,\theta}(a,b)
	=
	\mathcal H^{n-1}(D)\,(b-a)^{1-p}.
\end{equation}

In this model, the gradient does not degenerate and the fiber size does not vary with the level. Therefore, the reduced cost depends only on the separation $b-a$ between levels.

\begin{remark}
	To construct a concrete interior condenser inside $\Omega$, one can, for example, fix
	\[
	-L<a<b<c<L
	\]
	and take
	\[
	E=\{x\in\Omega:x_1\le a\},
	\qquad
	F=\{x\in\Omega:b\le x_1\le c\}.
	\]
This pair is no longer exactly of the form $(E_a,F_b)$, but it remains compatible with the same linear phase and preserves the same level structure on the interval $(a,b)$.
\end{remark}

\subsection{The radial model}
\label{subsec:modelo-radial}

Let
\[
\Omega=B(0,R)\subset\mathbb{R}^n,
\]
and consider the radial phase
\[
\theta(x)=|x|.
\]
Fix radii
\[
0<r_E<r_F<R,
\]
and let the levels of the phase determine the sets
\[
E_{r_E}=\{x\in\Omega:\theta(x)\le r_E\}=\overline{B(0,r_E)},
\qquad
F_{r_F}=\{x\in\Omega:\theta(x)\ge r_F\}=\{x\in\Omega:|x|\ge r_F\}.
\]
Here $E_{r_E}$ is strictly contained in $\Omega$, whereas $F_{r_F}$ in general touches the boundary $\partial\Omega$. Therefore, at this stage this pair is not regarded as an interior condenser in the sense fixed in the setup, but rather as an exact radial model of the reduced problem associated with the phase $\theta(x)=|x|$. To obtain an interior condenser compatible with the same radial structure, it is enough to truncate the outer plate and consider an annulus
\[
F=\{x\in\mathbb R^n:r_F\le |x|\le R_F\},
\qquad r_F<R_F<R.
\]
The subsection on exactness deals precisely with this configuration.

For each $t\in(0,R)$ one has
\[
\Sigma_t=\theta^{-1}(t)=\partial B(0,t),
\qquad
|\nabla\theta(x)|=1
\quad\text{for }x\neq 0.
\]
Consequently,
\[
S_\theta(t)=\mathcal H^{n-1}(\Sigma_t)=\omega_{n-1}t^{n-1},
\]
where
\[
\omega_{n-1}:=\mathcal H^{n-1}(S^{n-1}).
\]
Since moreover $|\nabla\theta|^{p-1}=1$ on each fiber, the energy weight coincides with the fiber size:
\begin{equation}
	\label{eq:peso-radial-ejemplo}
	A_{p,\theta}(t)=\omega_{n-1}t^{n-1},
	\qquad t\in(r_E,r_F).
\end{equation}

The reduced problem associated with this phase then takes the form
\begin{equation}
	\label{eq:funcional-radial}
	\mathcal E_{p,\theta}(v;r_E,r_F)
	=
	\omega_{n-1}\int_{r_E}^{r_F}|v'(t)|^p t^{n-1}\,\mathrm{d}t,
\end{equation}
for admissible functions $v$ such that
\[
v(r_E)=0,
\qquad
v(r_F)=1.
\]
Applying the exact formula for the reduced quantity, we obtain
\begin{equation}
	\label{eq:capacidad-reducida-radial}
	\operatorname{cap}_{p,\theta}(r_E,r_F)
	=
	\omega_{n-1}
	\left(
	\int_{r_E}^{r_F} t^{-\frac{n-1}{p-1}}\,\mathrm{d}t
	\right)^{1-p}.
\end{equation}

When $p\neq n$, the integral can be evaluated explicitly and yields
\begin{equation}
	\label{eq:capacidad-reducida-radial-pneqn}
	\operatorname{cap}_{p,\theta}(r_E,r_F)
	=
	\omega_{n-1}
	\left[
	\frac{p-1}{p-n}
	\left(
	r_F^{\frac{p-n}{p-1}}-r_E^{\frac{p-n}{p-1}}
	\right)
	\right]^{1-p}.
\end{equation}
In the critical case $p=n$, by contrast,
\begin{equation}
	\label{eq:capacidad-reducida-radial-peqn}
	\operatorname{cap}_{p,\theta}(r_E,r_F)
	=
	\omega_{n-1}
	\left(
	\log\frac{r_F}{r_E}
	\right)^{1-p}.
\end{equation}

In this model, the gradient of the phase remains regular, while the structure of the reduced cost is encoded in the variation of the size of the spherical fibers. In particular, the threshold $p=n$ separates a power-law regime from a logarithmic regime.

\begin{remark}
	At this stage, this example identifies the exact radial model of the reduced problem underlying the subsequent discussion of exactness. In particular, it identifies the energy weight
	\[
	A_{p,\theta}(t)=\omega_{n-1}t^{n-1}
	\]
	and the corresponding closed formula for
	\[
	\operatorname{cap}_{p,\theta}(r_E,r_F).
	\]
	The comparison with the full geometric capacity
	\[
	\operatorname{Cap}_{p,\Omega}(E_{r_E},F_{r_F})
	\]
	is postponed to the subsequent subsection.
\end{remark}

\subsection{The degenerate monomial model}
\label{subsec:modelo-monomial}

Let again
\[
\Omega=(-L,L)\times D\subset\mathbb{R}^n,
\]
with $L>0$ and $D\subset\mathbb{R}^{n-1}$ a bounded open set, and consider the phase
\[
\theta(x)=|x_1|^\gamma,
\qquad
\gamma>1.
\]
For $t>0$, the fibers are formed by two planar sections:
\[
\Sigma_t
=
\{x\in\Omega:x_1=t^{1/\gamma}\}
\cup
\{x\in\Omega:x_1=-t^{1/\gamma}\}.
\]
Moreover,
\[
|\nabla\theta(x)|=\gamma |x_1|^{\gamma-1}.
\]
On the level fiber $t$, this implies
\[
|\nabla\theta|\asymp t^{\frac{\gamma-1}{\gamma}}.
\]

Since each section has transversal area comparable to $\mathcal H^{n-1}(D)$, the total size of the fiber remains bounded above and below by positive constants independent of $t$. In particular,
\[
S_\theta(t)\asymp 1
\qquad
\text{as } t\to 0^+.
\]
Therefore,
\begin{equation}
	\label{eq:peso-monomial}
	A_{p,\theta}(t)
	\asymp
	t^{\frac{(\gamma-1)(p-1)}{\gamma}}
	\qquad
	\text{as } t\to 0^+.
\end{equation}
Consequently,
\begin{equation}
	\label{eq:resistencia-monomial}
	\left(A_{p,\theta}(t)\right)^{-\frac{1}{p-1}}
	\asymp
	t^{-\left(1-\frac{1}{\gamma}\right)}.
\end{equation}

Since
\[
1-\frac{1}{\gamma}<1,
\]
the local reduced resistance
\[
\int_0^\delta \left(A_{p,\theta}(t)\right)^{-\frac{1}{p-1}}\,\mathrm{d}t
\]
is finite for every sufficiently small $\delta>0$. Therefore, a monomial degeneracy of the gradient, by itself, does not force the collapse of the reduced capacity if the fiber size does not degenerate simultaneously.

Moreover, this example fits the local criterion of Section~\ref{sec:regimenes-criticos-theta}: here the exponent of gradient degeneracy is
\[
\alpha=1-\frac{1}{\gamma},
\]
whereas the geometric exponent of fiber collapse is
\[
\nu=0.
\]
Therefore,
\[
\alpha+\frac{\nu}{p-1}=1-\frac{1}{\gamma}<1,
\]
and the local regime remains transmissive.

\subsection{Symmetric models and exactness}
\label{subsec:modelos-simetricos-exactitud-clasica}

The following examples correspond to models in which the symmetry of the fibers makes it possible to verify explicitly that the reduction by a phase coincides with the full geometric capacity.

\label{subsubsec:modelo-plano-exacto}

In the planar model we fix levels
\[
-L<a<b<c<L
\]
and define the interior plates
\[
E:=\{x\in\Omega:x_1\le a\},
\qquad
F:=\{x\in\Omega:b\le x_1\le c\}.
\]
Then $(E,F)$ is an interior condenser in $\Omega$. By \eqref{eq:peso-plano} and \eqref{eq:capacidad-reducida-plano},
\[
A_{p,\theta}(t)=\mathcal H^{n-1}(D),
\qquad
\operatorname{cap}_{p,\theta}(a,b)
=
\mathcal H^{n-1}(D)\,(b-a)^{1-p}.
\]

\begin{proposition}[Exactness in the planar model]
	\label{prop:exactitud-modelo-plano}
	With the notation of the planar model,
	\[
	\operatorname{Cap}_{p,\Omega}(E,F)
	=
	\mathcal H^{n-1}(D)\,(b-a)^{1-p}
	=
	\operatorname{cap}_{p,\theta}(a,b).
	\]
\end{proposition}

\begin{proof}
	The inequality
	\[
	\operatorname{Cap}_{p,\Omega}(E,F)
	\le
	\operatorname{cap}_{p,\theta}(a,b)
	\]
	has already been obtained by fibered restriction.

	Let now
	\[
	u\in \mathcal A_{E,F}(\Omega).
	\]
	By Remark~\ref{rem:truncacion-clase-admisible}, we may replace \(u\) by its truncation
	\[
	T(u):=\min\{1,\max\{0,u\}\},
	\]
	without increasing the energy. We keep the notation \(u\) for this truncation. Then
	\[
	0\le u\le 1
	\qquad\text{a.e. in }\Omega.
	\]
	Since \(u\le 0\) q.e. on \(E\) and \(u\ge 1\) q.e. on \(F\), it follows that
	\[
	u=0 \quad\text{q.e. on }E,
	\qquad
	u=1 \quad\text{q.e. on }F.
	\]
	Because sets of zero Sobolev \(p\)-capacity have Lebesgue measure zero, these identities also hold a.e. on the plates. In particular,
	\[
	u(t,y)=0
	\quad\text{for a.e. }(t,y)\in E,
	\qquad
	u(t,y)=1
	\quad\text{for a.e. }(t,y)\in F.
	\]

	Define the transverse average
	\[
	\overline u(t)
	:=
	\frac{1}{\mathcal H^{n-1}(D)}
	\int_D u(t,y)\,\mathrm{d}y,
	\qquad t\in(-L,L).
	\]
	By Fubini, we obtain
	\[
	\overline u(t)=0
	\quad\text{for a.e. } t\le a,
	\qquad
	\overline u(t)=1
	\quad\text{for a.e. } b\le t\le c.
	\]
	In particular, the restriction of \(\overline u\) to \((a,b)\) is admissible for the planar reduced problem.

	Moreover, since \(u\in W^{1,p}((a,b)\times D)\), Fubini's theorem implies that for a.e. \(y\in D\) the section
	\[
	t\longmapsto u(t,y)
	\]
	belongs to \(W^{1,p}(a,b)\), with weak derivative \(\partial_1u(\cdot,y)\). Consequently, the transverse average \(\overline u\) belongs to \(W^{1,p}(a,b)\) and satisfies, for a.e. \(t\in(a,b)\),
	\[
	\overline u'(t)
	=
	\frac{1}{\mathcal H^{n-1}(D)}
	\int_D \partial_1u(t,y)\,\mathrm{d}y.
	\]

	By Jensen's inequality,
	\[
	|\overline u'(t)|^p
	\le
	\frac{1}{\mathcal H^{n-1}(D)}
	\int_D |\partial_1u(t,y)|^p\,\mathrm{d}y
	\le
	\frac{1}{\mathcal H^{n-1}(D)}
	\int_D |\nabla u(t,y)|^p\,\mathrm{d}y.
	\]
	Integrating over \((a,b)\), we get
	\[
	\mathcal H^{n-1}(D)\int_a^b |\overline u'(t)|^p\,\mathrm{d}t
	\le
	\int_{(a,b)\times D} |\nabla u(x)|^p\,\mathrm{d}x
	\le
	\int_\Omega |\nabla u(x)|^p\,\mathrm{d}x.
	\]

	On the other hand, by the explicit formula for the reduced quantity in the planar model,
	\[
	\mathcal H^{n-1}(D)\int_a^b |\overline u'(t)|^p\,\mathrm{d}t
	\ge
	\mathcal H^{n-1}(D)\,(b-a)^{1-p}.
	\]
	Therefore,
	\[
	\int_\Omega |\nabla u(x)|^p\,\mathrm{d}x
	\ge
	\mathcal H^{n-1}(D)\,(b-a)^{1-p}.
	\]
	Taking the infimum over all \(u\in\mathcal A_{E,F}(\Omega)\), we conclude that
	\[
	\operatorname{Cap}_{p,\Omega}(E,F)
	\ge
	\mathcal H^{n-1}(D)\,(b-a)^{1-p}.
	\]
	Together with the opposite inequality, this yields
	\[
	\operatorname{Cap}_{p,\Omega}(E,F)
	=
	\mathcal H^{n-1}(D)\,(b-a)^{1-p}
	=
	\operatorname{cap}_{p,\theta}(a,b).
	\]
\end{proof}

\label{subsubsec:modelo-radial-exacto}

We now fix radii
\[
0<r_E<r_F<R_F<R
\]
and define
\[
E:=\overline{B(0,r_E)},
\qquad
F:=\{x\in\mathbb R^n:r_F\le |x|\le R_F\}.
\]
Then $(E,F)$ is an interior condenser in $\Omega$. By \eqref{eq:peso-radial-ejemplo} and \eqref{eq:capacidad-reducida-radial},
\[
A_{p,\theta}(t)=\omega_{n-1}t^{n-1},
\qquad
\operatorname{cap}_{p,\theta}(r_E,r_F)
=
\omega_{n-1}
\left(
\int_{r_E}^{r_F} t^{-\frac{n-1}{p-1}}\,\mathrm{d}t
\right)^{1-p}.
\]

\begin{proposition}[Exactness in the radial model]
	\label{prop:exactitud-modelo-radial}
	With the notation of the radial model,
	\[
	\operatorname{Cap}_{p,\Omega}(E,F)
	=
	\omega_{n-1}
	\left(
	\int_{r_E}^{r_F} t^{-\frac{n-1}{p-1}}\,\mathrm{d}t
	\right)^{1-p}.
	\]
	In particular,
	\[
	\operatorname{Cap}_{p,\Omega}(E,F)
	=
	\operatorname{cap}_{p,\theta}(r_E,r_F).
	\]
\end{proposition}

\begin{proof}
	The inequality
	\[
	\operatorname{Cap}_{p,\Omega}(E,F)
	\le
	\operatorname{cap}_{p,\theta}(r_E,r_F)
	\]
	has already been obtained by fibered restriction, using admissible functions of the form
	\[
	u(x)=v(\theta(x))=v(|x|).
	\]

	Let now
	\[
	u\in \mathcal A_{E,F}(\Omega).
	\]
	By Remark~\ref{rem:truncacion-clase-admisible}, we may replace \(u\) by its truncation
	\[
	T(u):=\min\{1,\max\{0,u\}\},
	\]
	without increasing the energy. We keep the notation \(u\) for this truncation. Then
	\[
	0\le u\le 1
	\qquad\text{a.e. in }\Omega.
	\]
	Since \(u\le 0\) q.e. on \(E\) and \(u\ge 1\) q.e. on \(F\), it follows that
	\[
	u=0 \quad\text{q.e. on }E,
	\qquad
	u=1 \quad\text{q.e. on }F.
	\]
	Because sets of zero Sobolev \(p\)-capacity have Lebesgue measure zero, these identities also hold a.e. on the plates. In particular,
	\[
	u(x)=0
	\quad\text{for a.e. }x\in E,
	\qquad
	u(x)=1
	\quad\text{for a.e. }x\in F.
	\]

	Define the spherical average
	\[
	\overline u(r)
	:=
	\frac{1}{\omega_{n-1}r^{n-1}}
	\int_{\partial B(0,r)} u\,\mathrm{d}\mathcal H^{n-1},
	\qquad r\in(0,R).
	\]
	By Fubini in polar coordinates, it follows that
	\[
	\overline u(r)=0
	\quad\text{for a.e. } 0<r\le r_E,
	\qquad
	\overline u(r)=1
	\quad\text{for a.e. } r_F\le r\le R_F.
	\]
	In particular, the restriction of \(\overline u\) to \((r_E,r_F)\) is admissible for the radial reduced problem.

	Set
	\[
	A_{r_E,r_F}:=B(0,r_F)\setminus \overline{B(0,r_E)}.
	\]
	Since \(u\in W^{1,p}(A_{r_E,r_F})\) and this annulus is separated from the origin, there exists a sequence
	\[
	u_j\in C^\infty(\overline{A_{r_E,r_F}})
	\]
	such that
	\[
	u_j\to u
	\qquad\text{in }W^{1,p}(A_{r_E,r_F}).
	\]
	For each \(j\), define
	\[
	\overline u_j(r)
	:=
	\frac{1}{\omega_{n-1}r^{n-1}}
	\int_{\partial B(0,r)} u_j\,\mathrm{d}\mathcal H^{n-1},
	\qquad r\in(r_E,r_F).
	\]
	Then \(\overline u_j\) is smooth in \(r\) and satisfies
	\[
	\overline u_j'(r)
	=
	\frac{1}{\omega_{n-1}r^{n-1}}
	\int_{\partial B(0,r)} \partial_r u_j\,\mathrm{d}\mathcal H^{n-1}.
	\]
	Passing to the limit in \(W^{1,p}\) in polar coordinates, we obtain that \(\overline u\) belongs to the corresponding one-dimensional weighted space and that, for a.e. \(r\in(r_E,r_F)\),
	\[
	\overline u'(r)
	=
	\frac{1}{\omega_{n-1}r^{n-1}}
	\int_{\partial B(0,r)} \partial_r u\,\mathrm{d}\mathcal H^{n-1}.
	\]

	By Jensen's inequality,
	\[
	|\overline u'(r)|^p
	\le
	\frac{1}{\omega_{n-1}r^{n-1}}
	\int_{\partial B(0,r)} |\partial_r u|^p\,\mathrm{d}\mathcal H^{n-1}
	\le
	\frac{1}{\omega_{n-1}r^{n-1}}
	\int_{\partial B(0,r)} |\nabla u|^p\,\mathrm{d}\mathcal H^{n-1}.
	\]
	Multiplying by \(\omega_{n-1}r^{n-1}\) and integrating over \((r_E,r_F)\), we get
	\[
	\omega_{n-1}\int_{r_E}^{r_F} |\overline u'(r)|^p r^{n-1}\,\mathrm{d}r
	\le
	\int_{B(0,r_F)\setminus \overline{B(0,r_E)}} |\nabla u(x)|^p\,\mathrm{d}x
	\le
	\int_\Omega |\nabla u(x)|^p\,\mathrm{d}x.
	\]

	On the other hand, by the explicit formula for the reduced quantity in the radial model,
	\[
	\omega_{n-1}\int_{r_E}^{r_F} |\overline u'(r)|^p r^{n-1}\,\mathrm{d}r
	\ge
	\omega_{n-1}
	\left(
	\int_{r_E}^{r_F} t^{-\frac{n-1}{p-1}}\,\mathrm{d}t
	\right)^{1-p}.
	\]
	Therefore,
	\[
	\int_\Omega |\nabla u(x)|^p\,\mathrm{d}x
	\ge
	\omega_{n-1}
	\left(
	\int_{r_E}^{r_F} t^{-\frac{n-1}{p-1}}\,\mathrm{d}t
	\right)^{1-p}.
	\]
	Taking the infimum over all \(u\in\mathcal A_{E,F}(\Omega)\), we conclude that
	\[
	\operatorname{Cap}_{p,\Omega}(E,F)
	\ge
	\omega_{n-1}
	\left(
	\int_{r_E}^{r_F} t^{-\frac{n-1}{p-1}}\,\mathrm{d}t
	\right)^{1-p}.
	\]
	Together with the opposite inequality, this yields
	\[
	\operatorname{Cap}_{p,\Omega}(E,F)
	=
	\omega_{n-1}
	\left(
	\int_{r_E}^{r_F} t^{-\frac{n-1}{p-1}}\,\mathrm{d}t
	\right)^{1-p}
	=
	\operatorname{cap}_{p,\theta}(r_E,r_F).
	\]
\end{proof}

\begin{remark}
	In the two previous models, exactness comes from a complete symmetry of the fibers and from a transverse averaging argument that eliminates tangential oscillations without increasing the energy cost. One thus obtains a regime of exactness within the fibered framework.
\end{remark}

\subsection{Tangential obstruction outside the fibered regime}
\label{subsec:obstruccion-tangencial-fibrada}

The previous symmetric models show that the fibered reduction can be exact. Outside that regime, the inequality
\[
\operatorname{Cap}_{p,\Omega}(E,F)\le \operatorname{cap}_{p,\theta}(a,b)
\]
does not by itself imply exactness. In the linear case $p=2$, this separation is determined by the tangential energy of the minimizer relative to the fibers of $\theta$.

\begin{proposition}\label{prop:defecto-fibrado-lineal}
	Suppose that $p=2$. Let $(E,F)$ be an interior condenser in $\Omega$, and let
	\[
	\mathcal K_{E,F}(\Omega)
	:=
	\left\{
	u\in W^{1,2}(\Omega):
	\ u=0 \text{ q.e. on }E,\ 
	u=1 \text{ q.e. on }F
	\right\}.
	\]
	Let $u_*\in \mathcal K_{E,F}(\Omega)$ be a minimizer of the Dirichlet energy in $\mathcal K_{E,F}(\Omega)$, so that
	\[
	\operatorname{Cap}_{2,\Omega}(E,F)
	=
	\int_\Omega |\nabla u_*(x)|^2\,\mathrm{d}x.
	\]
	Suppose in addition that $u_f=v\circ \theta\in \mathcal K_{E,F}(\Omega)$ for some profile $v$, so that $u_f$ is an admissible fibered competitor. Then
	\[
	\int_\Omega |\nabla u_f(x)|^2\,\mathrm{d}x
	-
	\operatorname{Cap}_{2,\Omega}(E,F)
	=
	\int_\Omega |\nabla u_f(x)-\nabla u_*(x)|^2\,\mathrm{d}x.
	\]
	In particular, let
	\[
	R_\theta:=\{x\in\Omega:\nabla\theta(x)\neq 0\}.
	\]
	For a.e. $x\in R_\theta$, define
	\[
	\nu_\theta(x):=\frac{\nabla\theta(x)}{|\nabla\theta(x)|},
	\qquad
	\nabla_\perp u_*(x):=(\nabla u_*(x)\cdot \nu_\theta(x))\,\nu_\theta(x),
	\]
	and
	\[
	\nabla_\Sigma u_*(x):=\nabla u_*(x)-\nabla_\perp u_*(x).
	\]
	Then
	\[
	\int_\Omega |\nabla u_f(x)|^2\,\mathrm{d}x
	-
	\operatorname{Cap}_{2,\Omega}(E,F)
	\ge
	\int_{R_\theta} |\nabla_\Sigma u_*(x)|^2\,\mathrm{d}x.
	\]
\end{proposition}

\begin{proof}
	Since $u_*$ minimizes the energy in the affine class $\mathcal K_{E,F}(\Omega)$, it satisfies the Euler--Lagrange identity
	\[
	\int_\Omega \nabla u_*(x)\cdot \nabla \psi(x)\,\mathrm{d}x=0
	\]
	for every
	\[
	\psi\in W^{1,2}(\Omega)
	\quad\text{such that}\quad
	\psi=0 \text{ q.e. on }E\cup F.
	\]
	Since $u_f,u_*\in \mathcal K_{E,F}(\Omega)$, the difference
	\[
	\psi:=u_f-u_*
	\]
	belongs to $W^{1,2}(\Omega)$ and satisfies $\psi=0$ q.e. on $E\cup F$. Therefore,
	\[
	\int_\Omega \nabla u_*(x)\cdot \nabla (u_f-u_*)(x)\,\mathrm{d}x=0.
	\]
	Expanding, we obtain
	\[
	\int_\Omega \nabla u_*(x)\cdot \nabla u_f(x)\,\mathrm{d}x
	=
	\int_\Omega |\nabla u_*(x)|^2\,\mathrm{d}x.
	\]
	Using now the polarization identity,
	\[
	|\nabla u_f-\nabla u_*|^2
	=
	|\nabla u_f|^2+|\nabla u_*|^2-2\nabla u_f\cdot \nabla u_*,
	\]
	and integrating over $\Omega$, we obtain
	\[
	\int_\Omega |\nabla u_f-\nabla u_*|^2\,\mathrm{d}x
	=
	\int_\Omega |\nabla u_f|^2\,\mathrm{d}x
	-
	\int_\Omega |\nabla u_*|^2\,\mathrm{d}x.
	\]
	Since
	\[
	\int_\Omega |\nabla u_*|^2\,\mathrm{d}x
	=
	\operatorname{Cap}_{2,\Omega}(E,F),
	\]
	it follows that
	\[
	\int_\Omega |\nabla u_f|^2\,\mathrm{d}x
	-
	\operatorname{Cap}_{2,\Omega}(E,F)
	=
	\int_\Omega |\nabla u_f-\nabla u_*|^2\,\mathrm{d}x.
	\]
	
	Suppose now that $u_f=v\circ\theta$. Then $\nabla u_f$ is normal to the fibers of $\theta$ for a.e. $x\in R_\theta$. In the relevant region $R_\theta$, we decompose
	\[
	\nabla u_*
	=
	\nabla_\perp u_*+\nabla_\Sigma u_*,
	\]
	with $\nabla_\perp u_*$ normal to the fibers and $\nabla_\Sigma u_*$ tangential. Since $\nabla u_f$ has no tangential component in $R_\theta$,
	\[
	\nabla u_f-\nabla u_*
	=
	(\nabla u_f-\nabla_\perp u_*)-\nabla_\Sigma u_*,
	\]
	and the two components on the right-hand side are orthogonal at almost every point of $R_\theta$. Consequently,
	\[
	|\nabla u_f-\nabla u_*|^2
	=
	|\nabla u_f-\nabla_\perp u_*|^2
	+
	|\nabla_\Sigma u_*|^2
	\qquad\text{a.e. in }R_\theta.
	\]
Integrating and discarding the nonnegative term $ |\nabla u_f-\nabla_\perp u_*|^2 $, we obtain
	\[
	\int_\Omega |\nabla u_f|^2\,\mathrm{d}x
	-
	\operatorname{Cap}_{2,\Omega}(E,F)
	=
	\int_\Omega |\nabla u_f-\nabla u_*|^2\,\mathrm{d}x
	\ge
	\int_{R_\theta} |\nabla_\Sigma u_*|^2\,\mathrm{d}x.
	\]
\end{proof}

We denote by
\[
\mathcal A^{\mathrm{fib}}_{E,F}(\Omega;\theta)
:=
\left\{
u\in \mathcal K_{E,F}(\Omega):
\ u=v\circ\theta \text{ a.e. in }\Omega
\text{ for some profile }v
\right\}
\]
the fibered subclass of $\mathcal K_{E,F}(\Omega)$ associated with the phase $\theta$.

\begin{corollary}[Tangential control of the fibered defect]
	\label{cor:gap-control-tangencial}
	Under the hypotheses of the previous proposition, define
	\[
	\operatorname{Gap}_\theta(E,F)
	:=
	\inf\left\{
	\int_\Omega |\nabla u_f|^2\,\mathrm{d}x:
	u_f\in \mathcal A^{\mathrm{fib}}_{E,F}(\Omega;\theta)
	\right\}
	-
	\operatorname{Cap}_{2,\Omega}(E,F).
	\]
	Then
	\begin{equation}
		\label{eq:gap-control-tangencial}
		\operatorname{Gap}_\theta(E,F)
		\ge
		\int_{R_\theta} |\nabla_\Sigma u_*|^2\,\mathrm{d}x.
	\end{equation}
\end{corollary}

\begin{proof}
	The inequality \eqref{eq:gap-control-tangencial} holds for every admissible fibered function $u_f$. It is enough to take the infimum over that class.
\end{proof}

\begin{corollary}[Sufficient criterion for non-exact fibered reduction]
	\label{cor:criterio-suficiente-no-exactitud-fibrada}
	Suppose in addition that
	\[
	\mathcal A^{\mathrm{fib}}_{E,F}(\Omega;\theta)\neq\varnothing
	\]
	and that
	\[
	\int_{R_\theta} |\nabla_\Sigma u_*|^2\,\mathrm{d}x>0.
	\]
	Then
	\[
	\operatorname{Cap}_{2,\Omega}(E,F)
	<
	\inf\left\{
	\int_\Omega |\nabla u_f|^2\,\mathrm{d}x:
	u_f\in \mathcal A^{\mathrm{fib}}_{E,F}(\Omega;\theta)
	\right\}.
	\]
	In particular, the fibered reduction is not exact for that condenser and that phase.
\end{corollary}

\begin{proof}
	The conclusion is immediate from \eqref{eq:gap-control-tangencial}.
\end{proof}

\begin{remark}[Sufficient criterion for fibered exactness]
	\label{rem:criterio-suficiente-exactitud-fibrada}
	Under the hypotheses of the linear regime above, if the geometric minimizer $u_*$ belongs to the fibered subclass
	\[
	\mathcal A^{\mathrm{fib}}_{E,F}(\Omega;\theta),
	\]
	then
	\[
	\operatorname{Cap}_{2,\Omega}(E,F)
	=
	\inf\left\{
	\int_\Omega |\nabla u_f|^2\,\mathrm{d}x:
	u_f\in \mathcal A^{\mathrm{fib}}_{E,F}(\Omega;\theta)
	\right\},
	\]
	that is, the fibered reduction is exact for that condenser and that phase.
	
	In particular, the condition
	\[
	\nabla_\Sigma u_*=0
	\qquad\text{a.e. in }R_\theta
	\]
	serves here as a sufficient criterion when it is accompanied by an effective fibered representation of $u_*$ compatible with the plate conditions. By itself, the almost everywhere vanishing of the tangential component is not elevated in this work to an abstract global criterion of exactness.
\end{remark}

The previous results identify the following variational mechanism: in the linear case, every nonzero tangential component of the geometric minimizer $u_*$ relative to the fibers of $\theta$ produces a positive defect with respect to the fibered class. Thus, the exactness of the reduction requires at least that the geometric minimizer be, in the $L^2$ sense, purely normal to the level surfaces of the phase.

This reduction of the defect to a geometric condition on the linear minimizer identifies precisely the role of the tangential component in the failure of fibered exactness. In particular, a strict gap appears when the geometric minimizer does not admit a fibered representation with respect to the chosen phase, or equivalently, when its tangential component does not vanish.

\bibliographystyle{plain}
\bibliography{AbsCond}

\end{document}